\documentclass[10 pt,a4paper]{article}

\usepackage{amsfonts,amsmath,comment,verbatim,amsthm,amssymb,graphicx,fancyhdr,makeidx,amscd} 
\usepackage[all]{xy}
\newtheorem{definition}{Definition}
\newtheorem{proposition}[definition]{Proposition}
\newtheorem{lemma}[definition]{Lemma}
\newtheorem{theorem}[definition]{Theorem}
\newtheorem{remark}[definition]{Remark}
\newtheorem{corollary}[definition]{Corollary}

\begin{document}

\newcommand{\riem}{(M^m, \langle \, , \, \rangle)}
\newcommand{\Hess}{\mathrm{Hess}\, }
\newcommand{\hess}{\mathrm{hess}\, }
\newcommand{\cut}{\mathrm{cut}}
\newcommand{\ind}{\mathrm{ind}}

\newcommand{\ess}{\mathrm{ess}}

\newcommand{\longra}{\longrightarrow}

\newcommand{\eps}{\varepsilon}

\newcommand{\ra}{\rightarrow}

\newcommand{\vol}{\mathrm{vol}}

\newcommand{\di}{\mathrm{d}}

\newcommand{\R}{\mathbb R}

\newcommand{\C}{\mathbb C}

\newcommand{\Z}{\mathbb Z}

\newcommand{\N}{\mathbb N}

\newcommand{\HH}{\mathbb H}

\newcommand{\esse}{\mathbb S}

\newcommand{\bull}{\rule{2.5mm}{2.5mm}\vskip 0.5 truecm}

\newcommand{\binomio}[2]{\genfrac{}{}{0pt}{}{#1}{#2}} 

\newcommand{\metric}{\langle \, , \, \rangle}

\newcommand{\lip}{\mathrm{Lip}}

\newcommand{\loc}{\mathrm{loc}}

\newcommand{\diver}{\mathrm{div}}

\newcommand{\disp}{\displaystyle}

\newcommand{\rad}{\mathrm{rad}}

\newcommand{\Ricc}{\mathrm{Ric}}

\newcommand{\mmetric}{\langle\langle \, , \, \rangle\rangle}

\newcommand{\hol}{\mathrm{H\ddot{o}l}}

\newcommand{\capac}{\mathrm{cap}}

\newcommand{\bmo}{\{b <0\}}

\newcommand{\bmuo}{\{b \le 0\}}

\newcommand{\dist}{\mathrm{dist}}

\newcommand{\vp}{\varphi}

\renewcommand{\div}[1]{{\mathop{\mathrm div}}\left(#1\right)}

\newcommand{\divphi}[1]{{\mathop{\mathrm div}}\bigl(\vert \nabla #1

\vert^{-1} \varphi(\vert \nabla #1 \vert)\nabla #1   \bigr)}

\newcommand{\nablaphi}[1]{\vert \nabla #1\vert^{-1}

\varphi(\vert \nabla #1 \vert)\nabla #1}

\newcommand{\modnabla}[1]{\vert \nabla #1\vert }

\newcommand{\modnablaphi}[1]{\varphi\bigl(\vert \nabla #1 \vert\bigr)

\vert \nabla #1\vert }

\newcommand{\ds}{\displaystyle}

\newcommand{\cL}{\mathcal{L}}

\newcommand{\essem}{\mathds{S}^m}

\newcommand{\erre}{\mathds{R}}

\newcommand{\errem}{\mathds{R}^m}

\newcommand{\enne}{\mathds{N}}

\newcommand{\acca}{\mathds{H}}

\newcommand{\cvett}{\Gamma(TM)}

\newcommand{\cinf}{C^{\infty}(M)}

\newcommand{\sptg}[1]{T_{#1}M}

\newcommand{\partder}[1]{\frac{\partial}{\partial {#1}}}

\newcommand{\partderf}[2]{\frac{\partial {#1}}{\partial {#2}}}

\newcommand{\ctloc}{(\mathcal{U}, \varphi)}

\newcommand{\fcoord}{x^1, \ldots, x^n}

\newcommand{\ddk}[2]{\delta_{#2}^{#1}}

\newcommand{\christ}{\Gamma_{ij}^k}

\newcommand{\ricc}{\operatorname{Ricc}}

\newcommand{\supp}{\operatorname{supp}}

\newcommand{\sgn}{\operatorname{sgn}}

\newcommand{\rg}{\operatorname{rg}}

\newcommand{\inv}[1]{{#1}^{-1}}

\newcommand{\id}{\operatorname{id}}

\newcommand{\jacobi}[3]{\sq{\sq{#1,#2},#3}+\sq{\sq{#2,#3},#1}+\sq{\sq{#3,#1},#2}=0}

\newcommand{\lie}{\mathfrak{g}}

\newcommand{\wedgedot}{\wedge\cdots\wedge}

\newcommand{\rp}{\erre\mathds{P}}

\newcommand{\II}{\operatorname{II}}

\newcommand{\gradh}[1]{\nabla_{H^m}{#1}}

\newcommand{\absh}[1]{{\left|#1\right|_{H^m}}}

\newcommand{\mob}{\mathrm{M\ddot{o}b}}

\newcommand{\mab}{\mathfrak{m\ddot{o}b}}

\newcommand{\foc}{\mathrm{foc}}

\newcommand{\F}{\mathcal{F}}

\newcommand{\Cf}{\mathcal{C}_f}

\newcommand{\cutf}{\mathrm{cut}_{f}}

\newcommand{\Cn}{\mathcal{C}_n}

\newcommand{\cutn}{\mathrm{cut}_{n}}

\newcommand{\Ca}{\mathcal{C}_a}

\newcommand{\cuta}{\mathrm{cut}_{a}}

\newcommand{\cutc}{\mathrm{cut}_c}

\newcommand{\cutcf}{\mathrm{cut}_{cf}}

\newcommand{\rk}{\mathrm{rk}}

\newcommand{\crit}{\mathrm{crit}}

\newcommand{\diam}{\mathrm{diam}}

\newcommand{\haus}{\mathcal{H}}

\newcommand{\po}{\mathrm{po}}

\newcommand{\gr}{\mathcal{G}}

\newcommand{\sn}{\mathrm{sn}_H}

\newcommand{\cn}{\mathrm{cn}_H}

\newcommand{\Tr}{\mathrm{Tr}}

\newcommand{\bh}{\mathbb{B}}

\newcommand{\Pa}{\mathcal{P}}

\author{Alberto Farina${}^{(1,2)}$
\and Luciano Mari${}^{(3)}$
\and Enrico Valdinoci${}^{(4,5)}$ 
}
\title{\textbf{Splitting theorems, symmetry results \\
and overdetermined problems \\ for Riemannian manifolds}}
\date{}
\maketitle
\scriptsize \begin{center} (1) -- Laboratoire 
Ami\'enois de Math\'ematique Fondamentale et Appliqu\'ee\\
UMR CNRS 7352, Universit\'e Picardie ``Jules Verne'' \\33 Rue St Leu, 80039 Amiens (France).\\ 
\end{center}
\scriptsize \begin{center} (2) -- Institut ``Camille Jordan''\\
UMR CNRS 5208, Universit\'e
``Claude Bernard'' Lyon I\\
43 Boulevard du 11 novembre 1918, 69622 Villeurbanne cedex (France).\\
\end{center}  
\scriptsize \begin{center} (3) -- Departamento de Matem\'atica\\ 
Universidade Federal do Cear\'a\\
Campus do Pici, 60455-760 Fortaleza-CE (Brazil).
\end{center}  
\scriptsize \begin{center} (4) -- Dipartimento di Matematica
``Federigo Enriques''\\
Universit\`a
degli studi di Milano,\\
Via Saldini 50, I-20133 Milano (Italy).\\
\end{center}  
\scriptsize \begin{center} (5) -- 
Istituto di Matematica Applicata e Tecnologie Informatiche
``Enrico Magenes''\\
Consiglio Nazionale delle Ricerche\\
Via Ferrata 1, I-27100 Pavia (Italy).
\bigskip

E-mail addresses: alberto.farina@u-picardie.fr,
lucio.mari@libero.it, enrico.valdinoci@unimi.it

\end{center}

\normalsize

\vspace{0.8cm}
\begin{abstract}
Our work proposes a unified approach to three different topics in a general Riemannian setting: splitting theorems, symmetry results and overdetermined elliptic problems.  
By the existence of a stable solution to the semilinear equation $-\Delta u = f(u)$ on a Riemannian manifold with non-negative Ricci curvature, we are able to classify both the solution
and the manifold. We also discuss the classification of monotone (with respect to the direction of some Killing vector field) solutions, in the spirit of a conjecture of De Giorgi, and the rigidity features for overdetermined elliptic problems on submanifolds with boundary. 

\end{abstract}

\section*{Introduction and main results}
In this paper, we will study Riemannian
manifolds $(M, \metric)$ with non-negative Ricci curvature that possess a stable, nontrivial solution
of a semilinear equation of the type $-\Delta u = f(u)$. Under reasonable growth assumptions on $u$, we prove both
symmetry results for the solution and the rigidity of the underlying manifold. The case of
manifolds with boundary will be considered as well, in the framework of overdetermined problems. The main feature of our work is that we give a unified treatment, thereby providing a bridge beween three different topics in a general
Riemannian setting: splitting theorems, symmetry results and 
overdetermined problems. The key role here is played by a refined geometric Poincar\'e inequality, improving on those in \cite{sternbergzumbrun1, sternbergzumbrun2, Arma, Arma2}, see Proposition \ref{uguaglianzaintegrale} below. In the very particular case of Euclidean space, we recover previously known results in the literature.\par 
Firstly, we deal with complete, non-compact, boundaryless
Riemannian manifolds of non-negative Ricci curvature, that admit
a non-trivial stable solution. By assuming either a parabolicity condition 
or a bound on the energy growth, we obtain that the manifold splits off a factor $\R$ that completely determines the solution.
More precisely, we will prove

\begin{theorem}\label{teo_main}
Let $(M, \metric)$ be a complete, non-compact Riemannian manifold without boundary, satisfying $\Ricc \ge 0$. Suppose that $u \in C^3(M)$ be a non-constant, stable solution of $-\Delta u = f(u)$, for $f \in C^1(\R)$. If either
\begin{itemize}
\item[(i)] $M$ is parabolic and $\nabla u \in L^\infty(M)$, or 
\item[(ii)] The function $|\nabla u|$ satisfies 
\begin{equation}\label{condimportante}
\int_{B_R}|\nabla u|^2\di x = o (R^2 \log R) \qquad \text{as } \, R \ra +\infty.
\end{equation}
\end{itemize}
Then, 
\begin{itemize}
\item[-] $M= N \times \R$ with the product metric $\metric = \metric_N + 
\di t^2$, for some complete, totally geodesic, parabolic hypersurface $N$. In particular, $\Ricc^N \ge 0$ if $m \ge 3$, and $M =\R^2$ or $\esse^1 \times \R$, with their flat metric, if $m =2$; 
\item[-] $u$ depends only on $t$, has no critical points, and writing $u=y(t)$ it holds $y''=-f(y)$. 
\end{itemize}
Moreover, if $(ii)$ is met, 
\begin{equation}\label{volumepar}
\vol(B_R^N) = o(R^2 \log R) \qquad \text{as } \, R \ra +\infty.
\end{equation}
\begin{equation}\label{volumeint}
\int_{-R}^R |y'(t)|^2\di t \, =  o\Big( \frac{R^2 \log R}{\vol(B_R^N) } 
\Big) \qquad \text{as } \, R \ra +\infty.
\end{equation}
\end{theorem}

%
%
Basic facts on parabolicity can be found in \cite{grigoryan}, Sections 5 and 7.
We underline that, under a suitable sign assumption on~$f$, in 
Theorem~\ref{teo_pos} below we will obtain
that every stable solution is constant.\par
For our purposes, it is convenient to define $\mathcal{F}_2$ to be the family of complete manifolds $M$ with non-negative Ricci tensor
that, for each fixed $f \in C^1(\R)$, do not possess any stable, non-constant solution $u \in C^3(M)$ of $-\Delta u = f(u)$ for which
$$ 
\int_{B_R} |\nabla u|^2 \di x = o(R^2 \log R) \qquad \text{as } R \ra +\infty.
$$
Next Proposition \ref{prop_2e3} and Theorem~\ref{th13} give a complete classifications
of $M$ using this family:
\begin{proposition}\label{prop_2e3}
Let $(M^m, \metric)$ be a complete, non-compact manifold with $\Ricc \ge 0$. Then,
\begin{itemize}
\item[-] if $m=2$, $M \in \F_2$ if and only if $M$ is neither $\R^2$ nor $\esse^1 \times \R$ with their flat metric;
\item[-] if $m=3$, $M \in \F_2$ if and only if $M$ does not split off an Euclidean factor.
\end{itemize}
\end{proposition}
\begin{theorem}\label{th13}
Let $(M^m, \metric)$ be a complete, non-compact manifold with $\Ricc \ge 0$ and dimension $m\ge 3$. Suppose that $M \not \in \mathcal{F}_2$. Then, one and only one of the following possibilities occur:
\begin{itemize}
\item[{(i)}] $M=N^{m-1}\times \R$, where $N^{m-1} \in \F_2$ is either compact or it is parabolic, with only one end and with no Euclidean factor. Furthermore, 
\begin{equation}\label{inteparab}
\vol(B_R^N) = o(R^2 \log R) \qquad \text{as } R \ra +\infty.
\end{equation}
\item[{(ii)}] either $m=3$ and $M = \R^3$ or $\esse^1 \times \R^2$ with flat metric, or $m\ge 4$ and $ M = \bar N^{m-2} \times \R^2$, where $\bar N^{m-2} \in \F_2$ is either compact or it is parabolic, with only one end and with no Euclidean factor. Moreover, 
\begin{equation}\label{inteparab2}
\vol(B_R^{\bar N}) = o(R \log R) \qquad \text{as } R \ra +\infty.
\end{equation}
\item[{(iii)}] either $m=4$ and $M = \esse^1 \times \R^3$ with flat metric, or $m \ge 5 $ and $M = \hat N^{m-3} \times \R^3$, where $\hat N^{m-3}$ is compact with $\Ricc^{\hat N} \ge 0$. 
\end{itemize}
\end{theorem}

%
%
\begin{remark}\label{rem_liu}
\rm From the topological point of view, it has been recently proved in \cite{liu} that a non-compact, $3$-manifold with $\Ricc \ge 0$ is either diffeomorphic to $\R^3$ or its universal cover splits off a line (isometrically). This causes extra-rigidity for the manifolds $N, \bar N$ in the previous theorem. On the other hand, compact $3$-manifolds with $\Ricc \ge 0$ have been classified in \cite{hamilton4man} (Theorem 1.2) via Ricci flow techniques. Namely, they are diffeomorphic to a quotient of either $\esse^3$, $\esse^2\times \R$ or $\R^3$ by a group of fixed point free isometries in the standard metrics. 
\end{remark}
The case of manifolds with boundary will be considered here
in the light of overdetermined problems. 
In this spirit, Killing vector fields play a special role, as underlined by the next

\begin{theorem}\label{teo_mainbordo}
Let $(M, \metric)$ be a complete, non-compact  Riemannian manifold without boundary, satisfying  $\Ricc \ge 0$ and let $X$ be a Killing field on $M$. Let $\Omega \subseteq M$ be an open and connected set with $C^3$ boundary. Suppose that $u \in C^3(\overline\Omega)$ is a non-constant solution of the overdetermined problem 
\begin{equation}\label{pser}
\left\{ \begin{array}{ll} 
-\Delta u = f(u) & \quad \text{on } \Omega \\[0.1cm]
u= \, \mathrm{constant} & \quad \text{on } \partial \Omega \\[0.1cm]
\partial_\nu u = \, \mathrm{constant } \neq 0  & \quad \text{on } \partial \Omega.
\end{array}\right.
\end{equation}
Such that $\langle\nabla u, X\rangle$ is either positive or negative on $\Omega$. Then, if either
\begin{itemize}
\item[(i)] $M$ is parabolic and $\nabla u \in L^\infty(\Omega)$, or 
\item[(ii)] the function $|\nabla u|$ satisfies 
$$
\int_{\Omega \cap B_R}|\nabla u|^2\di x = o(R^2\log R) \qquad \text{as } \, R \ra +\infty,
$$
\end{itemize}
the following properties hold true:
\begin{itemize}
\item[-] $X$ is never zero, $\Omega = \partial \Omega \times \R^+$ with the product metric $\metric = \metric_{\partial \Omega} + \di t^2$, $\partial \Omega$ is totally geodesic in $M$ and satisfies $\Ricc_{\partial \Omega}  \ge 0$.

\item[-] the function $u$ depends only on $t$, it
has no critical points, and writing $u=y(t)$ it holds $y''=-f(y)$; 
\item[-] for every $t_0 \in \R$, the projected field $X^\perp = X - \langle X,\partial_t \rangle \partial_t$ at $(\cdot, t_0) \in \partial \Omega \times \{t_0\}$ is still a Killing field tangent to the fiber $\partial \Omega \times \{t_0\}$, possibly with singularities or identically zero;
\item[-] if $(ii)$ is met, $\partial \Omega$ satisfies $\vol(B_R^{\partial \Omega}) = o(R^2 \log R)$ as $R \ra +\infty$.
\end{itemize}
%
\end{theorem}

\begin{remark}
\emph{For $\R^2$, the above theorem generalizes the one-dimensional symmetry result in Theorem 1.2  of \cite{Arma}.
See also \cite{ros_sic}
for interesting studies on the
geometric and topological properties of overdetermined problems
in the Euclidean plane.} 
\end{remark}
By the monotonicity Theorem 1.1 in \cite{bercaffnire}, the relation $\langle \nabla u, X\rangle >0$ on $\Omega$ is authomatic for globally Lipschitz epigraphs $\Omega$ of Euclidean space and for some large class of nonlinearities $f$ including the prototype Allen-Cahn one $f(u) = u-u^3$ (even without requiring the Neumann condition in \eqref{pser}). However, it is an open problem to enlarge the class of domains $\Omega \subseteq \R^m$ for which $\langle \nabla u, X\rangle >0$ is met, or to find nontrivial analogues on Riemannian manifolds. In the last section, we move some steps towards this problem by proving some lemmata that may have independent interest. In particular, we obtain the next result:
\begin{proposition}\label{prop_elegante}
Let $(M, \metric)$ be a complete Riemannian manifold satisfying $\Ricc \ge -(m-1)H^2 \metric$, for some $H \ge 0$, and let $f \in C^1(\R)$ have the properties
$$
\left\{\begin{array}{l}
\disp f >0 \quad \text{on } (0, \lambda), \qquad f(\lambda)=0, \qquad f < 0 \quad \text{on } (\lambda, \lambda + s_0), \\[0.3cm]
\disp f(s) \ge \left(\delta_0 + \frac{(m-1)^2H^2}{4}\right)s \quad \text{for } s \in (0, s_0),
\end{array}\right.
$$
for some $\lambda>0$ and some small $\delta_0,s_0>0$. Let $\Omega \subseteq M$ be an open, connected subset, and suppose that $u \in C^2(\Omega) \cap C^0(\overline \Omega)$ is a bounded, non-negative solution of 
$$
\left\{ \begin{array}{l}
-\Delta u = f(u) \qquad \text{on } \Omega, \\[0.2cm]
u>0 \quad \text{on } \Omega, \qquad \sup_{\partial \Omega} u < \|u\|_{L^\infty(\Omega)}, \qquad  
\end{array}\right.
$$
Suppose that, for each $R$, $\Omega_R = \{ x\in \Omega \, : \, \dist(x,\partial \Omega) > R\}$ is non-empty. Then, the following properties hold:
\begin{equation}\label{propu}
\begin{array}{rl}
\rm{(I)} & \quad \|u\|_{L^\infty(\Omega)} = \lambda; \\[0.2cm]
\rm{(II)} & \quad \text{there exists a $R_0= R_0(m,H,\delta_0)>0$ such that, for each connected} \\[0.1cm]
& \quad \text{component $V_j$ of $\Omega_{R_0}$, } u(x) \ra \lambda \,  \text{ uniformly as } \ \dist (x, \partial \Omega) \ra +\infty \text{ along } V_j.
\end{array}
\end{equation}
\end{proposition}
\begin{remark}
\emph{We underline that $\partial \Omega$ may even have countably many connected components. Moreover, since $\Omega$ is possibly non-compact, $\Omega_{R_0}$ may have countably many connected components. In this respect, the uniformity guaranteed at point (II) is referred to each single, fixed connected component.  
}
\end{remark}

Although, as said, the general problem of ensuring the monotonicity of each solutions of the Dirichlet problem 
\begin{equation}\label{probdirichl}
\left\{ \begin{array}{l} 
-\Delta u = f(u) \qquad \text{on } \Omega \\[0.1cm]
u>0 \quad \text{on } \Omega, \qquad u= 0 \quad \text{on } \partial \Omega
\end{array}\right.
\end{equation}
is still open, for some ample class of nonlinearities we will be able to construct non-costant solutions of \eqref{probdirichl} which are strictly monotone, see the next Proposition \ref{prop_constrsol}. To do so, we shall restrict the set of Killing fields to the subclass described in the next
\begin{definition}\label{def_goodkilling}
Let $\Omega \subseteq M$ be an open, connected subset with $C^3$ boundary. A Killing vector field $X$ on $\overline \Omega$ is called good for $\Omega$ if its flow $\Phi_t$ satisfies 
\begin{equation}\label{ipokilling}
\left\{\begin{array}{l}
(i) \ \ \Phi_t(\Omega) \subseteq \Omega, \quad \Phi_t(\partial \Omega) \subseteq \overline\Omega \qquad \text{for every } \, t\in \R^+; \\[0.2cm] 
(ii) \ \ \text{there exists } \, o \in \partial \Omega \, \text{ for which } \quad  \dist \big( \Phi_t(o), \partial \Omega \big) \ra +\infty \quad \text{ as } \, t \ra +\infty.
\end{array}\right.
\end{equation}
\end{definition}
Observe that property $(i)$ is the somehow minimal requirement for investigating the monotonicity of $u$ with respect to $X$. The important assumption $(ii)$ enables us to insert arbitrarily large balls in $\Omega$, an essential requirement for our arguments to work. Under the existence of a good Killing field on $\Omega$, Proposition \ref{prop_elegante} allow us to produce some energy estimate via the method described in \cite{ambrosiocabre}, leading to the next particularization of Theorem \ref{teo_mainbordo} in the three dimensional case:

%
%
%
%

\begin{theorem}\label{teo_speciale}
Let $(M, \metric)$ be a complete Riemannian $3$-manifold, with empty boundary and with $\Ricc \ge 0$. Let $\Omega \subseteq M$ be an open, connected set with $\partial \Omega \in C^3$, fix $o \in \partial \Omega$ and assume that
\begin{equation}\label{ipovolume}
\haus^{2}(\partial \Omega \cap B_R) = o(R^2 \log R) \qquad \text{as } \, R \ra +\infty,
\end{equation}
where $B_R = B_R(o)$ and $\haus^2$ is the $2$-dimensional Hausdorff measure. Suppose that $\Omega$ has a good Killing field $X$. Let $f \in C^1(\R)$ be such that
\begin{equation*}
\left\{\begin{array}{l}
\disp f >0 \quad \text{on } (0, \lambda), \qquad f(\lambda)=0, \qquad f < 0 \quad \text{on } (\lambda, \lambda + s_0), \\[0.3cm]
\disp f(s) \ge \delta_0s \quad \text{for } s \in (0, s_0),
\end{array}\right.
\end{equation*}
for some $\lambda>0$ and some small $\delta_0,s_0>0$. If there exists a non-constant, positive, solution $u \in C^{3}(\overline\Omega)$ of the overdetermined problem 
\begin{equation}\label{pser.0}
\left\{ \begin{array}{ll} 
-\Delta u = f(u) & \quad \text{on } \Omega \\[0.1cm]
u= 0 & \quad \text{on } \partial \Omega \\[0.1cm]
\partial_\nu u = \, \mathrm{constant } & \quad \text{on } \partial \Omega,
\end{array}\right.
\end{equation}
such that 
\begin{equation}
\left\{\begin{array}{l}
\|u\|_{C^1(\Omega)} < +\infty; \\[0.1cm]
\langle X, \nabla u \rangle \ge 0 \qquad \text{on } \, \Omega, 
\end{array}\right.
\end{equation}
then all the conclusions of Theorem \ref{teo_mainbordo} hold.
\end{theorem}

\begin{remark}
\emph{Particularizing to $M = \R^3$ and for globally Lipschitz epigraphs $\Omega$, we recover Theorem 1.8 in \cite{Arma}, see Corollary \ref{teo_specialer3} below.}
\end{remark} 

\section*{Setting and notations}
Let $(M, \metric)$ be a smooth connected Riemannian manifold of dimension $m\ge 2$, without boundary. We briefly fix some notation. We denote with $K$ its sectional curvature and with $\Ricc$ its Ricci tensor. Having fixed an origin $o$, we set $r(x)= \mathrm{dist}(x,o)$, and we write $B_R$ for geodesic balls centered at $o$. If we need to emphasize the manifold under consideration, we will add a superscript $M$, so that, for instance, we will also write $\Ricc^M$ and $B_R^M$. The Riemannian $m$-dimensional volume will be indicated with $\vol$, and its density with $\di x$, while the will write $\haus^{m-1}$ for the induced $(m-1)$-dimensional Hausdorff measure. Throughout the paper, with the symbol $\{\Omega_j\}\uparrow M$ we mean a family $\{\Omega_j\}$, $j \in \N$, of relatively compact, open sets with smooth boundary and satisfying
$$
\Omega_j \Subset \Omega_{j+1} \Subset M, \qquad M = \bigcup_{j=0}^{+\infty} \Omega_j,
$$
where $A \Subset B$ means $\overline A \subseteq B$. Such a family will be called an exhaustion of $M$. 
Hereafter, we consider
\begin{equation}\label{definf}
f \in C^1(\R), 
\end{equation}
and a solution $u$ on $M$ of 
\begin{equation}\label{equazu}
-\Delta u = f(u) \qquad \text{on } M.
\end{equation}
\begin{remark}
\emph{
To avoid unessential technicalities, hereafter we assume that $u \in C^3(M)$. By standard elliptic estimates (see \cite{gilbargtrudinger}), $u \in C^3$ is automatic whenever $f \in C^{1,\alpha}_\loc(\R)$, for some $\alpha \in (0,1)$, and $u$ is a locally bounded weak solution of \eqref{equazu}. Analogously, for an open set $\Omega \subseteq M$ with boundary, we shall restrict to $u \in C^3(\overline \Omega)$. This condition is automatically satisfied whenever $\partial \Omega$ is, for instance, of class $C^3$.
}
\end{remark}
\begin{remark}
\emph{
For the same reason, we shall restrict to $f \in C^1(\R)$, although our statements could be rephrased for $f \in \lip_\loc(\R)$ with some extra-care in the definition of stability. In this respect, we suggest the reader to consult \cite{FSV} for a detailed discussion.
}
\end{remark}
%


We recall that $u$ is characterized, on each open subset $U \Subset M$, as a stationary point of the energy functional $E_U : H^1(U) \ra\R$ given by
\begin{equation}\label{energia}
E_U(w) = \frac 12 \int_{U} |\nabla w|^2 \di x - \int_U F(w) \di x, \qquad \text{where} \quad  F(t) = \int_0^t f(s) \di s,
\end{equation}
with respect to compactly supported variations in $U$. Let $J$ be the Jacobi operator of $E$ at $u$, that is, 
\begin{equation}
J\phi = -\Delta  \phi - f'(u) \phi \qquad \forall \, \phi \in C^\infty_c(M).
\end{equation}
\begin{definition}
The function $u$ solving \eqref{equazu} is said to be a \emph{\textbf{stable solution}} if $J$ is non-negative on $C^\infty_c(M)$, that is, if $(\phi, J\phi)_{L^2} \ge 0$ for each $\phi \in C^\infty_c(M)$. Integrating by parts, this reads as
\begin{equation}\label{hardy}
\int_{M} f'(u) \phi^2\di x \le \int_M |\nabla \phi|^2\di x \qquad \text{for every } \phi \in C^\infty_c(M).
\end{equation}
\end{definition}
By density, we can replace $C^\infty_c(M)$ in \eqref{hardy} with $\lip_c(M)$. By a result of \cite{fischercolbrieschoen} and \cite{mosspie} (see also \cite{prs}, Section 3) the stability of $u$ turns out to be equivalent to the existence of a positive $w\in C^1(M)$ solving $\Delta w + f'(u)w=0$ weakly on $M$.

%
%
%

\section*{Some preliminary computation}
We start with a Picone-type identity. 
\begin{lemma}\label{lem_picone}
Let $\Omega \subseteq M$ be an open, connected set with $C^3$ boundary (possibly empty) and let $u \in C^3(\overline{\Omega})$ be a solution of $-\Delta u =f(u)$ on $\Omega$. Let $w 
\in C^1(\overline{\Omega}) \cap C^2(\Omega)$ be a solution of $\Delta  w + f'(u)w \le 0$ such 
that $w>0$ on $\Omega$. Then, the following inequality holds true: for 
every $\eps>0$ and for every $\phi \in \lip_c(M)$,
\begin{equation}\label{eqpicone}
\begin{array}{lcl}
\disp \int_{\partial \Omega} \frac{\phi^2}{w+\eps} (\partial_\nu w)\di \haus^{m-1} & \le & \disp \int_\Omega |\nabla \phi|^2 \di x  -\int_\Omega f'(u) \frac{w}{w+\eps}\phi^2 \di x \\[0.5cm]
& & \disp - \int_\Omega (w+\eps)^2 \left|\nabla\left(\frac{\phi}{w+\eps}\right)\right|^2\di x
\end{array}
\end{equation} 
Furthermore, if either $\Omega = M$ or $w > 0$ on $\overline{\Omega}$, one can also take $\eps= 0$ inside the above inequality. The inequality is indeed an equality if $w$ solves $\Delta  w + f'(u)w=0$ on $\Omega$.
\end{lemma}

\begin{proof}
We integrate $\Delta w + f'(u)w \le 0$ against the test function $\phi^2/(w+\eps)$ to deduce
\begin{equation}\label{step1}
\begin{array}{lcl}
0 & \le  & \disp - \int_\Omega (\Delta w+ f'(u)w) \frac{\phi^2}{w+\eps}\di x = -\int_{\partial \Omega} \frac{\phi^2}{w+\eps}(\partial_\nu w)\di \haus^{m-1} \\[0.5cm]
& & \disp + \int_{\Omega} \langle \nabla\left(\frac{\phi^2}{w+\eps}\right), \nabla w\rangle\di x - \int_{\Omega} f'(u)w\frac{\phi^2}{w+\eps}\di x.
\end{array}
\end{equation}
Since
\begin{equation}
\langle \nabla\left(\frac{\phi^2}{w+\eps}\right), \nabla w\rangle  = 2\frac{\phi}{w+\eps}\langle \nabla \phi, \nabla w\rangle - \frac{\phi^2}{(w+\eps)^2}|\nabla w|^2,
\end{equation}
using the identity
\begin{equation}\label{picone}
\begin{array}{l}
\disp (w+\eps)^2 \left|\nabla\left(\frac{\phi}{w+\eps}\right)\right|^2 = |\nabla \phi|^2
+ \frac{\phi^2}{(w+\eps)^2}|\nabla w|^2 - 2 \frac{\phi}{w+\eps}\langle\nabla w,\nabla \phi\rangle 
\end{array}
\end{equation}
we infer that
\begin{equation}
\langle \nabla\left(\frac{\phi^2}{w+\eps}\right), \nabla w\rangle = |\nabla \phi|^2 - (w+\eps)^2 \left|\nabla\left(\frac{\phi}{w+\eps}\right)\right|^2.
\end{equation}
Inserting into \eqref{step1} we get the desired \eqref{eqpicone}.
\end{proof}

Next step is to obtain an integral equality involving the second derivatives of $u$. This geometric Poincar\'e-type formula has its roots in the paper \cite{Arma}, which deals with subsets of Euclidean space, and in the previous works \cite{sternbergzumbrun1, sternbergzumbrun2}. 
\begin{proposition}\label{uguaglianzaintegrale}
In the above assumptions, for every $\eps>0$ the following integral inequality holds true:
\begin{equation}\label{integrimpo}
\begin{array}{lcl}
\disp \int_\Omega \left[ |\nabla \di u|^2 + \Ricc(\nabla u, \nabla u)\right]\frac{\phi^2w}{w+\eps} \di x - \int_\Omega \big|\nabla |\nabla u|\big|^2 \phi^2 \di x  \\[0.5cm]
\disp \le \int_{\partial\Omega}\frac{\phi^2}{w+\eps} \left[w\partial_\nu \left(\frac{|\nabla u|^2}{2}\right) - |\nabla u|^2 \partial_\nu w \right]\di \haus^{m-1} \\[0.5cm]
 \ \ \ \disp + \eps \int_{\Omega} \frac{\phi}{w+\eps} \langle \nabla \phi, \nabla |\nabla u|^2 \rangle \di x - \frac 12 \int_\Omega \phi^2 \langle \nabla |\nabla u|^2, \nabla \left(\frac{w}{w+\eps}\right)\rangle \di x \\[0.5cm]
\ \ \ \disp + \int_\Omega |\nabla \phi|^2|\nabla u|^2 \di x - \int_\Omega (w+\eps)^2 \left|\nabla\left(\frac{\phi|\nabla u|}{w+\eps}\right)\right|^2\di x. 
\end{array}
\end{equation} 
Furthermore, if either $\Omega = M$ or $w > 0$ on $\overline{\Omega}$, one can also take $\eps= 0$. The inequality is indeed an equality if $\Delta w + f'(u)w = 0$ on $\Omega$.
\end{proposition}
\begin{proof}
We start with the B\"ochner formula 
\begin{equation}
\frac 12 \Delta |\nabla u|^2 = \langle \nabla \Delta u, \nabla u \rangle + \Ricc(\nabla u, \nabla u) + |\nabla \di u|^2,
\end{equation}
valid for each $u \in C^3(\overline \Omega)$. The proof of this formula is standard and can be deduced from Ricci commutation laws. Since $u$ solves $-\Delta  u = f(u)$, we get
\begin{equation}\label{bochnerdiff}
\frac 12 \Delta |\nabla u|^2 = -f'(u)|\nabla u|^2 + \Ricc (\nabla u, \nabla u) + |\nabla \di u|^2.
\end{equation}
Integrating \eqref{bochnerdiff} on $\Omega$ against the test function $\psi = \phi^2w/(w+\eps)$ we deduce
\begin{equation}\label{dopobochner}
\begin{array}{l}
\disp \int_\Omega \left[ |\nabla \di u|^2 + \Ricc (\nabla u, \nabla u)\right]\psi \di x 
 \\[0.5cm]
=\disp \int_\Omega f'(u)|\nabla u|^2\frac{w}{w+\eps}\phi^2\di x + 
\frac 12 
\int_\Omega \frac{w\phi^2}{w+\eps} \Delta  |\nabla u|^2\di x = \\[0.5cm]
\disp = \int_\Omega f'(u)|\nabla u|^2\frac{w}{w+\eps}\phi^2\di x + \frac 12 \int_{\partial\Omega}\frac{w\phi^2}{w+\eps} \partial_\nu |\nabla u|^2 \di \haus^{m-1} 
\\[0.5cm]
 \ \ \ \disp 
- \frac 12 \int_{\Omega}\langle \nabla \left(\frac{w\phi^2}{w+\eps}\right), \nabla |\nabla u|^2 \rangle \di x \\[0.5cm]
\disp = \int_\Omega f'(u)|\nabla u|^2\frac{w}{w+\eps}\phi^2\di x + \frac 12 \int_{\partial\Omega}\frac{w\phi^2}{w+\eps} \partial_\nu |\nabla u|^2 \di \haus^{m-1} \\[0.5cm]
 \ \ \ \disp - \int_{\Omega} \frac{w\phi}{w+\eps} \langle \nabla \phi, \nabla |\nabla u|^2 \rangle \di x - \frac 12 \int_\Omega \phi^2 \langle \nabla |\nabla u|^2, \nabla \left(\frac{w}{w+\eps}\right)\rangle \di x. 
\end{array}
\end{equation} 
Next, we consider the spectral inequality \eqref{eqpicone} with test function $\phi|\nabla u| \in \lip_c(M)$:  
\begin{equation}\label{laspettrale}
\begin{array}{lcl}
&&\disp \int_{\partial \Omega} |\nabla u|^2\frac{\phi^2}{w+\eps} 
(\partial_\nu w) \di \haus^{m-1} \\[0.5cm]
& \le & \disp \disp \int_\Omega |\nabla 
(\phi|\nabla u|)|^2 \di x - \int_\Omega f'(u) \frac{w}{w+\eps}|\nabla u|^2 \phi^2 \di x \\[0.5cm]
& & \disp - \int_\Omega (w+\eps)^2 \left|\nabla\left(\frac{\phi|\nabla u|}{w+\eps}\right)\right|^2\di x \\[0.5cm]
& = & \disp \disp \int_\Omega |\nabla \phi|^2|\nabla u|^2 \di x + \int_\Omega \phi^2 \big|\nabla |\nabla u|\big|^2\di x + 2 \int_\Omega \phi |\nabla u| \langle \nabla \phi, \nabla |\nabla u|\rangle \di x \\[0.5cm] 
& & \disp - \int_\Omega f'(u) \frac{w}{w+\eps}|\nabla u|^2 \phi^2 \di x - \int_\Omega (w+\eps)^2 \left|\nabla\left(\frac{\phi|\nabla u|}{w+\eps}\right)\right|^2\di x.
\end{array}
\end{equation}
Recalling that $\nabla |\nabla u|^2 = 2 |\nabla u|\nabla |\nabla u|$ weakly on $M$, summing up \eqref{laspettrale} and \eqref{dopobochner}, putting together the terms of the same kind and rearranging we deduce \eqref{integrimpo}. 
\end{proof}

\begin{proposition}\label{epsazero}
In the above assumptions, if it holds
\begin{equation}\label{assunz}
\liminf_{\eps \ra 0^+} \int_\Omega \phi^2 \langle \nabla |\nabla u|^2, \nabla \left(\frac{w}{w+\eps}\right)\rangle \di x  \ge 0,
\end{equation}
Then 
\begin{equation}\label{integrimpolim}
\begin{array}{lcl}
\disp \int_\Omega \left[ |\nabla \di u|^2 + \Ricc (\nabla u, \nabla u)- \big|\nabla |\nabla u|\big|^2 \right]\phi^2 \di x \\[0.5cm]
\disp + \liminf_{\eps \ra 0^+} \int_\Omega (w+\eps)^2 \left|\nabla\left(\frac{\phi|\nabla u|}{w+\eps}\right)\right|^2\di x \le\\[0.5cm]
\disp \le \int_\Omega |\nabla \phi|^2|\nabla u|^2 \di x + 
\liminf_{\eps 
\ra 0^+} \int_{\partial\Omega}\frac{\phi^2}{w+\eps} \left[w\partial_\nu \left(\frac{|\nabla u|^2}{2}\right) - |\nabla u|^2\partial_\nu w \right]\di \haus^{m-1}.
\end{array}
\end{equation}
\end{proposition}
\begin{proof}
We take limits as $\eps \ra 0^+$ in \eqref{integrimpo} along appropriate sequences. It is easy to see that
$$
\eps \int_{\Omega} \frac{\phi}{w+\eps} \langle \nabla \phi, \nabla |\nabla u|^2 \rangle \di x = o(1)
$$
as $\eps \ra 0$. Indeed, we can apply Lebesgue convergence theorem, since $|\eps/(w+\eps)| \le 1$, and we have convergence to the integral of the pointwise limit, which is:
$$
\int_{\{w=0\}} \phi\langle \nabla \phi, \nabla |\nabla u|^2 \rangle \di x = 0,
$$
being $\{w=0\} \subseteq \partial \Omega$. Lebesgue theorem can also be applied for the other terms in a straightforward way, with the exception of the term that needs \eqref{assunz}.
\end{proof}

Next, we need the following formula, that extends works of P. Sternberg and K. Zumbrun in \cite{sternbergzumbrun1, sternbergzumbrun2}. 
\begin{proposition}\label{sternzun}
Let $u$ be a $C^2$ function on $M$, and let $p\in M$ be a point such that $\nabla u(p) \neq 0$. Then, denoting with $|II|^2$ the second fundamental form of the level set $\Sigma=\{u = u(p)\}$ in a neighbourhood of $p$, it holds
$$
|\nabla \di u|^2 - \big|\nabla |\nabla u| \big|^2 = |\nabla u |^2 |II|^2 + \big|\nabla_T |\nabla u|\big|^2, 
$$
where $\nabla_T$ is the tangential gradient on the level set $\Sigma$.
\end{proposition}
\begin{proof}
Fix a local orthonormal frame $\{e_i\}$ on $\Sigma$, and let $\nu = \nabla u/|\nabla u|$ be the normal vector. For every vector field $X\in \Gamma(TM)$, 
$$
\nabla \di u (\nu, X) = \frac{1}{|\nabla u|} \nabla \di u (\nabla u, X) = \frac{1}{2|\nabla u|} \langle \nabla |\nabla u|^2, X\rangle = \langle \nabla |\nabla u|, X\rangle.
$$
Moreover, for a level set 
$$
II = -\frac{\nabla \di u_{|T\Sigma \times T\Sigma}}{|\nabla u|}.
$$
Therefore:
$$
\begin{array}{lcl}
\disp |\nabla \di u|^2 & = & \disp \sum_{i,j} \big[\nabla \di u (e_i,e_j)\big]^2 + 2\sum_j \big[\nabla \di u (\nu, e_j)\big]^2 + \big[\nabla \di u(\nu, \nu)\big]^2 \\[0.3cm]
& = & \disp |\nabla u|^2 |II|^2 + 2 \sum_j \big[\langle \nabla |\nabla u|, e_j\rangle\big]^2 + \big[ \langle \nabla |\nabla u|, \nu\rangle\big]^2 \\[0.3cm]
& = & \disp |\nabla u|^2 |II|^2 + \big|\nabla_T |\nabla u|\big|^2 + \big|\nabla |\nabla u|\big|^2. 
\end{array}
$$
proving the proposition.
\end{proof}

\section*{Splitting and structure theorems: the boundaryless case}
Our first result deals with the case when $\Omega = M$ has no boundary. It is inspired by the ones proved in \cite{FAR-H, FSV} for the Euclidean case, and also extends and strengthens some previous work in~\cite{FarSirVal}.
\begin{proof}[\textbf{Proof of Theorem~\ref{teo_main}}]
In our assumption, we consider the integral formula \eqref{integrimpo} with $M=\Omega$ and $\eps =0$.  Since $\Ricc \ge 0$ we deduce
\begin{equation}\label{inteprimo}
\disp \int_M \left[ |\nabla \di u|^2 - \big|\nabla |\nabla u|\big|^2\right]\phi^2\di x  \le \int_M |\nabla \phi|^2|\nabla u|^2 \di x - \int_M w^2 \left
|\nabla\left(\frac{\phi|\nabla u|}{w}\right)\right|^2\di x.
\end{equation}
Next, we rearrange the RHS as follows: using the inequality
$$
|X+Y|^2 \ge |X|^2 + |Y|^2 -2|X||Y| \ge (1-\delta)|X|^2 + (1-\delta^{-1})|Y|^2,
$$
valid for each $\delta>0$, we obtain
\begin{equation}\label{young}
\begin{array}{lcl}
\disp w^2\left|\nabla\left(\frac{\phi|\nabla u|}{w}\right)\right|^2 & = & \disp w^2\left| \frac{|\nabla u|\nabla\phi}{w} + \phi\nabla \left(\frac{|\nabla u|}{w}\right)\right|^2 \\[0.4cm]
& \ge & \disp (1-\delta^{-1}) |\nabla u|^2 |\nabla \phi|^2 + (1-\delta) \phi^2w^2\left|\nabla \left(\frac{|\nabla u|}{w}\right)\right|^2.
\end{array}
\end{equation}
Substituting in \eqref{inteprimo} yields
\begin{equation}\label{inteprimo2}
\disp \int_M \left[ |\nabla \di u|^2 - \big|\nabla |\nabla u|\big|^2\right]\phi^2\di x  + (1-\delta) \int_M \phi^2w^2 \left|\nabla\left(\frac{|\nabla u|}{w}\right)\right|^2\di x \le \frac 1 \delta \int_M |\nabla \phi|^2|\nabla u|^2 \di x.
\end{equation}
Choose $\delta <1$. We claim that, for suitable families $\{\phi_\alpha\}_{\alpha \in I \subseteq \R^+}$, it holds
\begin{equation}\label{claim}
\{\phi_\alpha\} \text{ is monotone increasing to 1,} \qquad \lim_{\alpha \ra +\infty} \int_M |\nabla \phi_\alpha|^2|\nabla u|^2 \di x = 0.
\end{equation}
Choose $\phi$ as follows, according to the case.
\begin{itemize}
\item[] In case $(i)$, fix $\Omega \Subset M$ with smooth boundary
and let $\{\Omega_j\} \uparrow M$ be a smooth exhaustion with $\Omega \Subset \Omega_1$. Choose $\phi = \phi_j \in \lip_c(M)$ to be identically $1$ on $\Omega$, $0$ on $M\backslash \Omega_j$ and the harmonic capacitor on $\Omega_j \backslash \Omega$, that is, the solution of 
$$
\left\{\begin{array}{l}
\Delta \phi_j = 0 \qquad \text{on } \Omega_j \backslash \Omega,\\[0.2cm]
\phi_j=1 \quad \text{on } \partial \Omega, \qquad \phi_j = 0 \quad \text{on } \partial \Omega_j.
\end{array}\right.
$$
Note that $\phi_j \in \lip_c(M)$ is ensured by elliptic regularity up to $\partial \Omega$ and $\partial \Omega_j$. By comparison and since $M$ is parabolic, $\{\phi_j\}$ is monotonically increasing and pointwise convergent to $1$, and moreover  
$$
\int_{\Omega_j} |\nabla \phi_j|^2|\nabla u|^2\di x \le \|\nabla 
u\|^2_{L^\infty} \mathrm{cap}(\Omega,\Omega_j) \ra \|\nabla 
u\|^2_{L^\infty} \mathrm{cap}(\Omega) = 0,
$$
the last equality following since $M$ is parabolic. This proves \eqref{claim}.
\item[] In case $(ii)$, we apply a logarithmic cutoff argument. For fixed $R>0$, choose the following radial $\phi(x) = \phi_R(r(x))$:
\begin{equation}\label{logcutoff}
\phi_R(r) = \left\{ \begin{array}{ll}
1 & \quad \text{if } r \le \sqrt{R}, \\[0.1cm]
\disp 2 - 2\frac{\log r}{\log R} & \quad \text{if } r \in [\sqrt{R}, R], \\[0.1cm]
0 & \quad \text{if } r \ge R.
\end{array}\right. 
\end{equation}
Note that 
$$
|\nabla \phi(x)|^2 = \frac{4}{r(x)^2\log^2R} \chi_{B_R\backslash B_{\sqrt{R}}}(x), 
$$
where $\chi_A$ is the indicatrix function of a subset $A\subseteq M$. Choose $R$ in such a way that $\log R/2$ is an integer. Then, 
\begin{equation}\label{lapr}
\begin{array}{l}
\disp \int_M |\nabla \phi|^2 |\nabla u|^2\di x  = \disp \int_{B_R 
\backslash B_{\sqrt{R}}} |\nabla \phi|^2 |\nabla u|^2\di x = 
\frac{4}{\log^2R} \sum_{k= \log R/2}^{\log R-1} \int_{B_{e^{k+1}}\backslash B_{e^k}} \frac{|\nabla u|^2}{r(x)^2}\di x \\[0.5cm]
\le \disp \frac{4}{\log^2R} \sum_{k= \log R/2}^{\log R} \frac{1}{e^{2k}} 
\int_{B_{e^{k+1}}} |\nabla u|^2 \di x.
\end{array}
\end{equation}
By assumption, 
$$
\int_{B_{e^{k+1}}} |\nabla u|^2 \di x \le (k+1)e^{2(k+1)} \delta(k)
$$
for some $\delta(k)$ satisfying $\delta(k) \ra 0$ as $k\ra +\infty$. Without loss of generality, we can assume $\delta(k)$ to be decreasing as a funtion of $k$. Whence,
\begin{equation}\label{lasec}
\begin{array}{l}
\disp \frac{4}{\log^2R} \sum_{k= \log R/2}^{\log R} \frac{1}{e^{2k}} \int_{B_{e^{k+1}}} |\nabla u|^2 \di x
\le \frac{8}{\log^2R} \sum_{k= \log R/2}^{\log R} \frac{e^{2(k+1)}}{e^{2k}}(k+1)\delta(k) \\[0.5cm]
\disp \le \frac{8e^2}{\log^2R}\delta(\log R/2) \sum_{k=0}^{\log R} (k+1) \le \frac{C}{\log^2R}\delta(\log R/2) \log^2 R =  C \delta(\log R/2),
\end{array}
\end{equation}
for some constant $C>0$. Combining \eqref{lapr} and \eqref{lasec} and letting $R\ra +\infty$ we deduce \eqref{claim}.
\end{itemize}
In both the cases, we can infer from the integral formula that
\begin{equation}\label{relazsplit}
|\nabla u|=cw, \ \text{ for some } c \ge 0,  \qquad |\nabla \di u|^2 = \big|\nabla |\nabla u|\big|^2, \qquad \Ricc(\nabla u, \nabla u)=0.
\end{equation}
Since $u$ is non-constant by assumption, $c>0$, thus $|\nabla u|>0$ on $M$. From B\"ochner formula, it holds
$$
|\nabla u| \Delta |\nabla u| + \big|\nabla |\nabla u|\big|^2 = \frac 12 \Delta |\nabla u|^2 = \Ricc(\nabla u, \nabla u) + |\nabla \di u|^2 - f'(u) |\nabla u|^2
$$
on $M$. Using \eqref{relazsplit}, we thus deduce that $\Delta |\nabla u| + f'(u)|\nabla u| = 0$ on $M$, hence $|\nabla u|$ (and so $w$) both solve the linearized equation $Jv = 0$.\\
Now, the flow $\Phi$ of $\nu =\nabla u/|\nabla u|$ is well defined on $M$. Since $M$ is complete and $|\nu|=1$ is bounded, $\Phi$ is defined on $M\times \R$. By \eqref{relazsplit} and Proposition \ref{sternzun}, $|\nabla u|$ is constant on each connected component of a level set $N$, and $N$ is totally geodesic. Therefore, in a local Darboux frame $\{e_j,\nu\}$ for the level surface $N$, 
\begin{equation}\label{consekato}
\begin{array}{lcl}
0 = |II|^2 \Longrightarrow \nabla \di u (e_i,e_j) = 0 \\[0.2cm]
0 = \langle \nabla |\nabla u|, e_j \rangle = \nabla \di u (\nu, e_j), 
\end{array}
\end{equation}
so the unique nonzero component of $\nabla \di u$ is that corresponding to the pair $(\nu,\nu)$. Let $\gamma$ be any integral curve of $\nu$. Then
$$
\frac{\di}{\di t} (u\circ \gamma) = \langle \nabla u, \nu\rangle = |\nabla u| \circ \gamma > 0
$$ 
and
$$
\begin{array}{lcl}
\disp -f(u\circ \gamma) & = & \disp \Delta u (\gamma) = \nabla \di u(\nu,\nu)(\gamma) = \langle \nabla |\nabla u|, \nu\rangle (\gamma) \\[0.2cm]
& = & \disp \frac{\di}{\di t} (|\nabla u|\circ \gamma) = \frac{\di^2}{\di t^2} (u\circ \gamma),
\end{array}
$$
thus $y = u\circ \gamma$ solves the ODE $y'' = -f(y)$ and $y' >0$.  Note also that the integral curves $\gamma$ of $\nu$ are geodesics. Indeed,
$$
\begin{array}{lcl}
\disp \nabla_{\gamma'} \gamma' & = & \disp \frac{1}{|\nabla u|} \nabla_{\nabla u} \left(\frac{\nabla u}{|\nabla u|}\right) =  \frac{1}{|\nabla u|^2} \nabla_{\nabla u}\nabla u - \frac{1}{|\nabla u|^3} \nabla u(|\nabla u|) \nabla u \\[0.4cm]
& = & \disp \frac{1}{|\nabla u|^2} \nabla \di u(\nabla u,\cdot)^\sharp - \frac{1}{|\nabla u|^3} \langle \nabla |\nabla u|, \nabla u\rangle \nabla u \\[0.4cm]
& = &\disp \frac{1}{|\nabla u|} \nabla \di u(\nu,\cdot)^\sharp - \frac{1}{|\nabla u|} \langle \nabla |\nabla u|, \nu \rangle \nu =  \frac{1}{|\nabla u|} \nabla \di u(\nu,\cdot)^\sharp - \frac{1}{|\nabla u|} \nabla \di u(\nu,\nu) \nu =  0,
\end{array}
$$
where the first equality in the last line follows from \eqref{consekato}. We now address the topological part of the splitting, following arguments in the proof of \cite{prs}, Theorem 9.3. Since $|\nabla u|$ is constant on level sets of $u$, $|\nabla u| = \beta(u)$ for some function $\beta$. Evaluating along curves $\Phi_t(x)$, since $u \circ \Phi_t$ is a local bijection we deduce that $\beta$ is continuous. We claim that $\Phi_t$ moves level sets of $u$ to level sets of $u$. Indeed, integrating $\di/\di s (u \circ \Phi_s) = |\nabla u| \circ \Phi_s = \alpha(u\circ \Phi_s)$ we get
$$
t = \int_{u(x)}^{u(\Phi_t(x))} \frac{\di \xi}{\alpha(\xi)},
$$
thus $u(\Phi_t(x))$ is independent of $x$ varying in a level set. As $\alpha(\xi) >0$, this also show that flow lines starting from a level set of $u$ do not touch the same level set. Let $N$ be a connected component of a level set of $u$. Since the flow of $\nu$ is through geodesics, for each $x\in N$ $\Phi_t(x)$ coincides with the normal exponential map $\exp^\perp(t \nu(x))$. Moreover, since $N$ is closed in $M$ and $M$ is complete, the normal exponential map is surjective: indeed, by variational arguments, each geodesic from $x\in M$ to $N$ minimizing $\dist(x,N)$ is perpendicular to $N$. This shows that $\Phi_{|N\times \R}$ is surjective. We now prove the injectivity of $\Phi_{|N\times \R}$. Suppose that $\Phi(x_1,t_1)= \Phi(x_2,t_2)$. Then, since $\Phi$ moves level sets to level sets, necessarily $t_1=t_2=t$. If by contradiction $x_1 \neq x_2$, two distinct flow lines of $\Phi_t$ would intersect at the point $\Phi_t(x_1)=\Phi_t(x_2)$, contradicting the fact that $\Phi_t
 $ is a diffeomorphism on $M$ for every $t$. Concluding, $\Phi : N \times \R \ra M$ is a diffeomorphism. In particular, each level set $\Phi_t(N)$ is connected. This proves the topological part of the splitting.\\
We are left with the Riemannian part. We consider the Lie derivative of the metric in the direction of $\Phi_t$:
$$
\begin{array}{lcl}
\disp \big(L_\nu \metric\big)(X,Y) & = & \langle \nabla_X \nu,Y\rangle + \langle X, \nabla_Y \nu\rangle \\[0.4cm]
& = & \disp \frac{2}{|\nabla u|} \nabla \di u (X,Y) + X\left(\frac{1}{|\nabla u|}\right) \langle \nabla u, Y\rangle + Y \left(\frac{1}{|\nabla u|}\right) \langle \nabla u, X\rangle.
\end{array}
$$
From the expression, using that $|\nabla u|$ is constant on $N$ and the properties of $\nabla \di u$ we deduce that 
$$
\big(L_\nu \metric\big)(X,Y) = \frac{2}{|\nabla u|} \nabla \di u (X,Y) = 0.
$$
If at least one between $X$ and $Y$ is in the tangent space of $N$. If, however, $X$ and $Y$ are normal, (w.l.o.g. $X=Y=\nabla u$), we have
$$
\begin{array}{lcl}
\disp \big(L_\nu \metric\big)(\nabla u, \nabla u) & = & \disp \frac{2}{|\nabla u|} \nabla \di u (\nabla u, \nabla u) + 2\nabla u\left(\frac{1}{|\nabla u|}\right) |\nabla u|^2 \\[0.4cm]
& = & \disp \frac{2}{|\nabla u|} \nabla \di u (\nabla u, \nabla u) - 2\nabla u(|\nabla u|) = 2\nabla \di u (\nu, \nabla u) - 2\langle \nabla |\nabla u|, \nabla u\rangle = 0.
\end{array}
$$
Concluding, $L_\nu \metric = 0$, thus $\Phi_t$ is a flow of isometries. Since $\nabla u \perp TN$, $M$ splits as a Riemannian product, as desired. In particular, $\Ricc^N \ge 0$ if $m\ge 3$, while, if $m=2$, $M =\R^2$ or $\esse^1 \times \R$ with the flat metric.\par 
We next address the parabolicity. Under assumption $(i)$, $M$ is parabolic and so $N$ is necessarily parabolic too. We are going to deduce the same under assumption $(ii)$. To this end, it is enough to prove the volume estimate \eqref{volumepar}. Indeed, \eqref{volumepar} is a sufficient condition on $M$ to be parabolic.
The chain of inequalities
$$ 
\Big(\int_{-R}^R |y'(t)|^2\di t \, \Big)\vol(B_R^N) \le \int_{[-R,R] \times B_R^N} |y'(t)|^2\di t \, \di x^N  \le \int_{B_{R\sqrt 2}} |\nabla u|^2 \di x = o (R^2 \log R) 
$$
gives immediately \eqref{volumepar} and \eqref{volumeint},  since $\vert y' \vert>0 $ everywhere. 
\end{proof}

\begin{remark}\label{rem_intemeglio}
\emph{The proof of Theorem \ref{teo_main} is tightly related to some works in \cite{prs} and \cite{pigolaveronelliIJM}, see in particular Theorems 4.5 and 9.3 in \cite{prs}, and Theorem 4 in \cite{pigolaveronelliIJM}. We note that, however, our technique is different from the one used to prove the vanishing Theorem 4.5 in \cite{prs}. Namely, this latter is based on showing that $|\nabla u|/w$ is a weak solution of the inequality
$$
\Delta_{w^2} \left(\frac{|\nabla u|}{w}\right) \ge 0 \qquad \text{on } M, \text{ where} \quad \Delta_{w^2} = w^{-2}\diver(w^2\nabla\cdot),
$$
and then concluding via a refined Liouville-type result that improves on works of \cite{bercaffnire2} (Theorem 1.8) and \cite{ambrosiocabre} (Proposition 2.1). However, this approach seems to reveal some difficulties when dealing with sets $\Omega$ having non-empty boundary, thereby demanding a different method. Our technique, which uses from the very beginning the spectral inequality \eqref{eqpicone}, is closer in spirit to the one in \cite{pigolaveronelliIJM}.
}
\end{remark}
Under suitable sign assumptions on $f$, Theorem \ref{teo_main} implies a Liouville type result thanks to a Caccioppoli-type estimate. This is the content of the next 
\begin{theorem}\label{teo_pos} Let $(M, \metric)$ be a complete, non-compact manifold with $\Ricc \ge 0$  and let~$u \in C^3(M)$ be a bounded stable solution of $-\Delta u = f(u)$ on M, with $f \in C^1(\R) $ and 
\begin{equation}\label{Si}
f(r) \ge 0 \qquad \text{for any } r\in \R.
\end{equation}
Suppose that either $m \le 4$ or   
\begin{equation}\label{vag}  
\vol(B_R) = o(R^4 \log R) \qquad \text{as } \, R \ra +\infty.
\end{equation}
Then, $u$ is constant.

\end{theorem}

\begin{proof} The proof is by contradiction.  Suppose that $u$ is not constant and set $u^* = \sup_Mu$.  Then, multiplying the equation by  $\phi^2 (u^*-u)$  and integrating by parts we get
$$
\begin{array}{lcl}
\disp 0 & \ge & \disp \int_M |\nabla u|^2 \phi^2 \di x - 2 \int_M \phi(u^*-u) \langle \nabla u, \nabla \phi\rangle \di x \\[0.4cm]
& \ge & \disp \int_M |\nabla u|^2 \phi^2 \di x - 4\|u\|_{L^\infty} \int_M \phi |\nabla u||\nabla \phi| \di x,
\end{array}
$$
thus, by Young inequality, there exists $C = C(\|u\|_{L^\infty})>0$ such that 
$$
\frac 12 \int_M |\nabla u|^2 \phi^2 \di x \le C \int_M |\nabla \phi|^2 \di x. 
$$
Considering a radial $\phi(x) = \phi_R(r(x))$, where $\phi_R(r)$ satisfies 
$$
\phi_R(r) = 1 \quad \text{on } [0,R], \qquad \phi_R(r) = \frac{2R-r}{R} \quad \text{on } [R, 2R], \qquad \phi_R(r)=0 \quad \text{on } [2R,+\infty)
$$
we get
%
%
$$
\int_{B_R} |\nabla u|^2 \di x \le \int_M |\nabla u|^2 \phi^2 \di x \le C \vol(B_{2R} ) R^{-2} 
$$
where $C>0$ is a constant independent of $R$. 
When $ m\le 4$, we have  $\vol(B_{2R}) \le C' R^4$ (for some constant $C'>0$ independent of $R$) by Bishop-Gromov volume comparison theorem, thus condition \eqref{condimportante} in Theorem 
\ref{teo_main} is satisfied.  On the other hand, \eqref{condimportante} is always satisfied when $m\ge5$ and \eqref{vag} are in force.  Therefore, by Theorem \ref{teo_main}, $u =y(t) $ solves  
$$
-y''=f(y) \ge 0
$$
hence $y$, being nonconstant, must necessarily be unbounded, a contradiction that concludes the proof.   
\end{proof}
%
%
%
%
%
%
%


Theorem \ref{teo_main} can be iterated to deduce the structure Theorem \ref{th13}. To do so, we define the following families: 
\begin{itemize}

\item[] $\F_1 = \big\{$complete, parabolic manifolds $(M, \metric)$ with $\Ricc \ge 0$, admitting no stable, non-constant solutions $u \in C^3(M)$ of $-\Delta u = f(u)$ with $|\nabla u| \in L^\infty(M)$, for any $f \in C^1(\R)\big\}$.

\item[] $\F_2 = \big\{$complete manifolds $(M, \metric)$ with $\Ricc \ge 0$, admitting no stable, non-constant solutions $u \in C^3(M)$ of $-\Delta u = f(u)$ for which
$$
\int_{B_R} |\nabla u|^2 \di x = o(R^2 \log R) \qquad \text{as } R \ra +\infty,
$$
for any $f \in C^1(\R) \big\}$.
\end{itemize}

\noindent The next result is an immediate consequence of Theorem \ref{teo_main}, and improve upon previous works in \cite{dupfar, FarSirVal}.

%
%

%

%
%
\begin{corollary}\label{cor_ricciquasipos}
If $(M, \metric)$ be complete, non-compact Riemannian manifold. Suppose that $M$ has quasi-positive Ricci curvature, that is, that $\Ricc \ge 0$ and $\Ricc_x > 0$ for some point $x \in M$. Then, $M \in \F_2$. If $m=2$, we also have $ M \in \F_1$.
\end{corollary}
\begin{proof}
Otherwise, by Theorem \ref{teo_main}, $M= N \times \R$ with the product metric $\metric = \metric_N + \di t^2$ (if $m=2$, by Bishop-Gromov volume comparison $\vol(B_R) \le \pi R^2$, so $M$ is parabolic).
Therefore, $\Ricc(\partial_t,\partial_t)=0$ at every point $x = (\bar x, t) \in M$, contradicting the quasi-positivity assumption.
\end{proof}
\begin{remark}
\emph{The above conclusion  is sharp.  Indeed, $\R^2$ equipped  with its canonical flat metric is parabolic and supports the function $u(x,y) = x$, which is a non-constant, harmonic function, hence a non-constant stable solution of \eqref{equazu}  with $f=0$.
} 
\end{remark}
\begin{remark} \label{compatta}
\emph{By results in \cite{Jimbo, FSV}, any compact manifold  $(M, \metric)$  with $\Ricc \ge 0$ belongs to $\F_1\cap\F_2$.}
\end{remark}

To proceed with the investigation of $\F_1$, $\F_2$, we need a preliminary computation.
\begin{proposition}
Let $X$ be a vector field on $(M^m, \metric)$, and let $u \in C^3(M)$ be a solution of $-\Delta u = f(u)$, for some $f \in C^1(\R)$. Set for convenience $T = \frac 12 L_X \metric$. Then, the function $w = \langle \nabla u, X\rangle$ solves
\begin{equation}\label{gradxline}
\Delta w + f'(u)w = 2 \langle \nabla \di u, T \rangle + \big[ 2\diver(T) - \di \Tr(T) \big](\nabla u). 
\end{equation}
In particular, if $X$ is conformal, that is, $L_X \metric = \eta \metric$, for some $\eta \in C^\infty(M)$, then
$$
\Delta w + f'(u)w = -\eta f(u) + (2-m) \langle \nabla \eta, \nabla u \rangle.
$$
\end{proposition}
\begin{proof}
Fix a local orthonormal frame $\{e_i\}$, with dual coframe $\{\theta^j\}$. Let $R_{ijkt}$ be the components of the $(4,0)$ curvature tensor, with the standard sign agreement. We have
$$
X= X^ke_k, \quad w = u_kX^k, \quad \nabla X = X^k_i \theta^i\otimes e_k, \quad \nabla \di u = u_{ki} \theta^i \otimes \theta^k.
$$
For notational convenience, we lower all the indices with the aid of the metric $\metric = (g_{jk})$. Note that, for $X$, the lowered index is in the first position, that is, $X_{ki} = g_{kr}X^r_i$. According to the definition of $T$,
$$
T_{ik} = \frac 12 \left(L_X \metric\right)_{ik} = \frac 12 
\left(X_{ik}+X_{ki}\right).
$$
Then, since $u \in C^3(M)$,
\begin{equation}\label{calcoliprel}
\begin{array}{lcl}
\Delta w & = & \disp (u_kX_k)_{ii} = u_{ki,i}X_k + 2u_{ki}X_{ki} + u_k X_{ki,i} \\[0.2cm]
& = & \disp u_{ki,i}X_k + 2u_{ki}T_{ki} + u_k X_{ki,i},
\end{array}
\end{equation}
where the equality in the last row follows since $\nabla \di u$ is symmetric, whence only the symmetric part of $\nabla X$ survives. From Ricci commutation laws
\begin{equation}\label{riccicomm}
u_{rk,i} = u_{ri,k} + u_t R_{trki}, \qquad X_{rk,i} = X_{ri,k} + X_t R_{trki}, 
\end{equation}
Schwartz symmetry for the second derivatives of $u$, and the equality $X_{ii}=T_{ii}$ we deduce
\begin{equation}\label{lericciappl}
\begin{array}{lcl}
\disp u_{ki,i} & = & \disp \disp u_{ik,i} = u_{ii,k} + u_t R_{tiki} = (\Delta u)_k + u_t \Ricc_{tk}, \\[0.2cm]
\disp X_{ki,i} & = & \disp (2T_{ki} - X_{ik})_i = 2 T_{ki,i} - X_{ik,i} = 2 T_{ki,i} - X_{ii,k} - X_t R_{tiki} \\[0.2cm]
& = & \disp \disp 2 T_{ki,i} - T_{ii,k} - X_t \Ricc_{tk}
\end{array}
\end{equation}
Using \eqref{lericciappl} in \eqref{calcoliprel} we infer that
\begin{equation}\label{alfinalquasi}
\begin{array}{lcl}
\Delta w & = & \disp (\Delta u)_kX_k + u_t \Ricc_{tk}X_k + 2 u_{ki}T_{ki} + u_k\big( 2T_{ki,i} - T_{ii,k}\big) - u_k X_t \Ricc_{tk} \\[0.2cm]
& = & \disp - f'(u) u_kX_k + 2 u_{ki}T_{ki} + 2u_k T_{ki,i} - u_k T_{ii,k},
\end{array}
\end{equation}
and \eqref{gradxline} follows at once.
\end{proof}
An immediate application of the strong maximum principle (\cite{gilbargtrudinger, protterweinberger}) yields the following corollary
\begin{corollary}\label{stablesimple}
In the assumptions of the above theorem, if $X$ is a Killing vector field, then $w=\langle \nabla u, X\rangle$ is a solution of the Jacobi equation
$$
Jw = -\Delta w - f'(u)w = 0.
$$
In particular, if $w \ge 0$ on $M$, then either $w \equiv 0$ on $M$ or $w>0$ on $M$. Therefore, if a solution $u \in C^{3}(M)$ of $-\Delta u = f(u)$ is weakly monotone in the direction of some Killing vector field, then either $u$ is stable and strictly monotone in the direction of $X$, or $u$ is constant on the flow lines of $X$. 
\end{corollary}
With the aid of Corollary \ref{stablesimple}, we can prove the next results:
\begin{lemma}\label{lem_easy}
Let $M = N \times \R$ be a Riemannian product with $\Ricc \ge 0$. 
\begin{itemize}
\item[\emph{(I)}] If $M$ is parabolic, then $M \not \in \F_1$.
\item[\emph{(II)}] If $\vol(B_R^N) = o(R^2 \log R)$,  then  $M \not \in \F_2$. 
\end{itemize}
\end{lemma}
\begin{proof}
Denote the points of $M$ with $(x,t)$. Choose $f(t) = t-t^3$, and 
$$
u(x,t) =  \tanh\left(\frac{t}{\sqrt{2}}\right).
$$
Then, $u$ is a non-constant, globally Lipschitz solution of $-\Delta u = f(u)$, monotonic in the direction of the Killing field $\partial_t$. Thus, $u$ is stable by Corollary \ref{stablesimple}, and (I) immediately follows. Since 
$$
\int_{B_R} |\nabla u|^2 \di x \le \int_{[-R,R] \times B_R^N} |\partial_tu|^2 \di t \di x^N \le \|\partial_tu\|^2_{L^2(\R)} \vol(B_R^N), 
$$
$M \not \in \F_2$ provided that $\vol(B_R^N) = o(R^2 \log R)$, which shows (II).
\end{proof}
\begin{proposition}\label{prop_basef1f2}
Denote with $\Pa = \{\text{parabolic manifolds}\}$. Let $m$ be the dimension of the family of manifolds under consideration. Then
\begin{itemize}
\item[$(i)$] $\F_1 \subseteq \F_2$ for every $m\ge 2$;
\item[$(ii)$] $\F_2 \cap \Pa \subseteq \F_1$ for $m=2,3$;
\item[$(iii)$] $\F_1=\F_2$ for $m=2$, and $\F_2 \cap \Pa = \F_1$ for $m=3$;
\item[$(iv)$] $\F_1 \subsetneqq \F_2$ for every $m\ge 3$.
\end{itemize}
\end{proposition}
\begin{proof}
$(i)$. Suppose that $M \in \F_1 \backslash \F_2$. Then, by Remark \ref{compatta} $M$ is non-compact and thus, by Theorem \ref{teo_main}, $M = N \times \R$. Since $M$ is parabolic, by Lemma \ref{lem_easy} we conclude that $M \not \in \F_1$, a contradiction.\\
$(ii)$. Let $M \in \F_2 \cap \Pa$. If by contradiction $M \not \in \F_1$, then $M$ is non-compact, $M = N \times \R$ and $\Ricc^N \ge 0$ again by Theorem \ref{teo_main}. By Bishop-Gromov theorem, $\vol(B_R^N) \le CR^{m-1}$. If $m=2,3$, $N$ satisfies assumption (II) of Lemma \ref{lem_easy}, and so $M \not \in \F_2$, contradiction.\\
$(iii)$. By definition, $\F_1 \cap \Pa = \F_1$. Thus, from $(i)$ and $(ii)$, if $m \le 3$ it holds $\F_1 = \F_1 \cap \Pa \subseteq \F_2 \cap \Pa \subseteq \F_1$, hence $\F_2 \cap \Pa = \F_1$. On the other hand, if $m=2$, condition $\Ricc \ge 0$ and Bishop-Gromov comparison theorem imply that $M$ is parabolic, thus $\F_2 \cap \Pa = \F_2$ and so $\F_1 = \F_2$.\\
$(iv)$. In view of Corollary \ref{cor_ricciquasipos}, it is enough to produce a non-parabolic manifold with $\Ricc \ge 0$ and $\Ricc >0$ somewhere. For instance, we can take a model manifold $M_g$, that is, $\R^m$ equipped with a radially symmetric metric $\di s^2$ whose expression, in polar geodesic coordinates centered at some $o$, reads $\di s^2 = \di r^2 + g(r)^2 \metric_{\esse^{m-1}}$, $\metric_{\esse^{m-1}}$ being the standard metric on the unit sphere, with the choice
$$
g(r) = \frac r 2 + \frac 12 \arctan(r)
$$
(see \cite{grigoryan} or \cite{petersen} for basic formulae on radially symmetric manifolds). Standard computations show that $\Ricc >0$ outside $o$ and $\vol(\partial B_r) \ge C r^{m-1}$, hence $[\vol(\partial B_R)]^{-1} \in L^1(+\infty)$ for each $m \ge 3$, which is a sufficient condition for a model to be non-parabolic (see \cite{grigoryan}, Corollary 5.6.). Therefore, $M_g \not \in \F_1$, as required. 
\end{proof}
\begin{proof}[Proof of Proposition \ref{prop_2e3}]
It follows straightforwardly from Theorem \ref{teo_main}, Remark \ref{compatta} and Lemma \ref{lem_easy}. 
\end{proof}
%
%
%
%
%
%
We are now ready to prove Theorem \ref{th13}.
\begin{proof}[\textbf{Proof of Theorem \ref{th13}}] Since $M \not \in \F_2$, by Theorem \ref{teo_main} we have $M = N \times \R$, for some complete, parabolic $N$ with $\Ricc^N \ge 0$ satisfying the growth estimate \eqref{inteparab}.  If $N \in \F_2$ and it is non-compact, then by Lemma \ref{lem_easy} it has no Euclidean factor and we are in case $(i)$. In particular, $N$ has only one end, for otherwise it would contain a line and would split off an Euclidean factor according to Cheeger-Gromoll splitting theorem (see \cite{cheegergromoll} or \cite{petersen}, Theorem 68). Suppose that $(i)$ is not satisfied, hence $N \not\in \F_2$.  Then, if $m=3$ we have, by Theorem \ref{teo_main}, that $N$ is flat and $M = N \times \R = C \times \R^2$ for some curve $C$, thus $M= \R^3$ or $\esse^1 \times \R^2$ with a flat metric.  On the other hand, when $m\ge4$, we have $N=\bar N \times \R$ and by \eqref{inteparab} we deduce \eqref{inteparab2}. The same analysis performed for $N$ can now be
  repeated verbatim to $\bar N$ in order to obtain the desired conclusion. If $(ii)$ does not hold, then also $\bar N$ splits off a line, and $\bar N = \hat N \times \R$. If $m=4$, $\bar N$ is a flat surface and $\hat N$ is a curve, and the sole possibility to satisfy \eqref{inteparab2} is that $\hat N= \esse^1$ is closed. If $m \ge 5$, again by \eqref{inteparab2}, we deduce that
$$
\vol(B_R^{\hat N}) = o(\log R) \qquad \text{as } \, R \ra +\infty.
$$
By the Calabi-Yau growth estimate (see \cite{calabi_vol} and \cite{yau2}) a non-compact manifold with non-negative Ricci curvature has at least linear volume growth, and this forces $\hat N$ to be compact, concluding the proof. 
\end{proof}
\section*{An extended version of a conjecture of De Giorgi}

We consider an extended version (to Riemannian manifolds with $\Ricc \ge 0$) of a celebrated conjecture of E. De Giorgi. Let us recall that in 1978 E. De Giorgi \cite{DG} formulated the following question : 

{ \it Let~$u\in C^2(\R^m, [-1,1])$ satisfy
\begin{equation}\label{PDE}
-\Delta u= u-u^3\qquad\mbox{ and }\qquad\frac{\partial u}{\partial x_m}>0\qquad\mbox{ on } \R^m.
\end{equation}

Is it true that 
all the level sets of~$u$ are hyperplanes, at least if $m\le 8$? }

\noindent The original conjecture has been proven in dimensions $m=2, 3$ and it is still open, in its full generality, for $4 \le m \le 8$. We refer the reader to \cite{le} for a recent review on the conjecture of De Giorgi and related topics.\par
In our setting, we replace the (Euclidean) monotonicity assumption $ \partial u/\partial x_m>0$  on $\R^m$  by the natural one: $u$ is monotone with respect to the flow lines of some Killing vector field, and we investigate the geometry of the level set of $u$ as well as the symmetry properties of $u$.  This supplies a genuine framework for the study of the above conjecture on Riemannian manifolds. Our conclusion will be that the level sets of $u$ are complete, totally geodesics submanifolds of $M$, which is clearly the analogous in our context of  the classic version of De Giorgi's conjecture.  Our results apply to Riemannian manifolds with $\Ricc \ge 0$. In particular, they recover and improve the results concerning the Euclidean cases of  $\R^2$ and $\R^3$ and they also give a description of those manifolds supporting a De Giorgi-type conjecture.

\begin{theorem}\label{teo_DeGiorgi}
Let $(M, \metric)$ be a complete non-compact Riemannian manifold without boundary with $\Ricc \ge 0$ and let $X$ be a Killing field on $M$. Suppose that $u \in C^3(M)$  is a solution of 
\begin{equation*}
\left\{ \begin{array}{ll} 
-\Delta u = f(u) & \quad \text{on } M, \\[0.2cm]
\langle\nabla u, X\rangle>0   & \quad \text{on }  M,
\end{array}\right.
\end{equation*}
with $f \in C^1(\R)$.
If either
\begin{itemize}
\item[(i)] $M$ is parabolic and $\nabla u \in L^\infty(M)$ or 
\item[(ii)] the function $|\nabla u|$ satisfies 
$$
\int_{B_R}|\nabla u|^2\di x = o(R^2\log R) \qquad \text{as } \, R \ra +\infty,
$$
\end{itemize}
then, $M= N \times \R$ with the product metric $\metric = \metric_N + \di t^2$, for some complete, totally geodesic, parabolic submanifold $N$. In particular, $\Ricc^N \ge 0$ if $m \ge 3$, while, if $m=2$, $M =\R^2$ or $\esse^1 \times \R$ with their flat metric. 

Furthemore, $u$ depends only on $t$ and writing $u=y(t)$ it holds 
$$ -y''=f(y), \qquad y'>0.$$


\end{theorem}

\begin{proof}
Thanks to Corollary \ref{stablesimple}, $u$ is a non-constant stable solution of the considered equation. The desired conclusion is then a consequence of Theorem \ref{teo_main}.
\end{proof}


\begin{remark}\label{rem_lira}
\emph{
We spend few words to comment on possible topological and geometric restrictions coming from the monotonicity assumption. Condition $\langle \nabla u, X\rangle >0$ implies that both $\nabla u$ and $X$ are nowhere vanishing, hence $M$ is foliated by the smooth level sets of $u$. However, there is no a-priori Riemannian splitting. Similarly, the presence of the nowhere-vanishing Killing vector field $X$ on $M$ does not force, a-priori, any topological splitting of $M$ along the flow lines of $X$, as the orthogonal distribution $\mathcal{D}_X : x \, \mapsto \, X(x)^\perp$ is not automatically integrable for Killing fields. Therefore, the monotonicity requirement alone does not imply, in general, severe geometric restrictions. However, one should be careful that, when $\mathcal{D}_X$ is integrable and $X$ is Killing, the local geometry of $M$ then turns out to be quite rigid. Indeed, coupling the Frobenius integrability condition for $\mathcal{D}_X$ with the skew-symmetry of $\nabla X$ coming from the Killing condition, one checks that each leaf of $\mathcal{D}_X$ is totally geodesic.
%
Since $|X|$ is constant along the integral lines of $X$, locally in a neighbourhood of a small open subset $U\subseteq N$ the metric splits as the warped product 
$$
\metric = \metric_U + h(x)^{-2} \di t^2, \qquad \text{where } (x,t) \in U\times \R, 
$$
for some smooth $h(x) = |X|(x)^{-1}$. In particular, the Ricci curvature in the direction of $X=\partial_t$ satisfies
\begin{equation}\label{eqricc}
\Ricc(\partial_t, \partial_t) = - \frac{\Delta h (x)}{h(x)}.
\end{equation}
Further restrictions then come out when one adds the requirement $\Ricc \ge 0$. In this case, by \eqref{eqricc} $h(x)$ turns out to be a positive, superharmonic function on $U$. Consequently, if in a (global) warped product $N \times \R$, with metric $\metric_N + h^{-2}\di t^2$, the factor $N$ is parabolic, then condition $\Ricc \ge 0$ forces $N \times \R$ to be a Riemannian product, $h$ being constant by the parabolicity assumption.
The dimensional case $m=2$ is particularly rigid. In fact, if $M$ is a complete surface with non-negative Gaussian curvature and possessing a nowhere vanishing Killing vector field $X$, then $M$ is flat. Indeed, in this case $\mathcal{D}_X$ is clearly integrable, and the integral curves of the local unit vector field $E$ orthogonal to $X$ are geodesics. For $x \in M$, let $\sigma : \R \ra M$ be a unit speed geodesic with tangent vector everywhere orthogonal to $X$. The sectional curvature along $\sigma(t)$ is 
$$
0 \le K( \sigma' \wedge X) = \frac{R(\sigma',X,\sigma',X)(t)}{|X|^2(\sigma(t))} = -h(t) h''(t),
$$
so $h$ is a non-negative, concave function on $\R$, hence $h$ is constant. Therefore, $K =0$ along $\sigma$, and in particular at $x$, as claimed. Note that the completeness assumption on $M$ is essential, as the example of the punctured paraboloid $M = \{(x,y,z) \in \R^3 : z = x^2 + y^2, \, z > 0\}$ shows. 
}
\end{remark}

Theorem \ref{teo_DeGiorgi} has some interesting consequences. For instance, in the 2-dimensional case we have the following strengthened version:

\begin{corollary}\label{cor_DeGiorgi2D}
Let $(M, \metric)$ be a complete non-compact surface without boundary, with Gaussian curvature $K \ge 0$ and let $X$ be a Killing field on $M$. Suppose that $u \in C^3(M)$  is a solution of 
\begin{equation*}
\left\{ \begin{array}{ll} 
-\Delta u = f(u) & \quad \text{on } M \\[0.1cm]
\langle\nabla u, X\rangle>0   & \quad \text{on }  M\\[0.1cm]
\nabla u \in L^\infty(M)
\end{array}\right.
\end{equation*}
with $f \in C^1(\R)$.

Then, $M$ is the Riemannian product $\R^2$ or $\esse^1 \times \R$, with flat metric,  $u$ depends only on $t$ and, writing $u=y(t)$, it holds
$$ y''=-f(y), \qquad y'>0.
$$
\end{corollary}

\begin{proof}  Since $K \ge 0$ and $\dim(M) =2$, by Bishop-Gromov comparison theorem $\vol(B_R) \le \pi R^2$, so $M$ is parabolic by Theorem 7.3 in \cite{grigoryan}. Therefore, both (i) and (ii) of Theorem \ref{teo_DeGiorgi} are satisfied. This proves the corollary.
\end{proof}

Some remarks are in order.

\begin{remark} 

\rm
\begin{itemize}

\item[(i)] The previous result establishes De Giorgi's conjecture for surfaces with non-negative Gaussian curvature. Actually it yields more, indeed, if $(M, \metric)$ is a complete non-compact manifold  without boundary, with $\Ricc \ge 0$ and of dimension $m\ge 2$, it is known that any bounded solution of $-\Delta u = f(u)$ also has bounded gradient 
(see e.g. Appendix 1). Note also that the converse is not true, since $u(x) = x_1$, is an unbounded monotone harmonic function on 
$(\R^m, \metric_{\mathrm{can}}) $ whose gradient is bounded (here, and in the sequel, $\metric_{\mathrm{can}}$ denotes the canonical flat metric on $\R^m$).

\item[(ii)]  We recover the case of $\R^2$, with  its canonical flat metric. Apply Corollary \ref{cor_DeGiorgi2D} to $(\R^2, \metric_{\mathrm{can}})$ and $ X = \partial/\partial x_2$.

\item[(iii)] From Theorem \ref{teo_DeGiorgi} we also recover the case of $(\R^3, \metric_{\mathrm{can}})$. Indeed, any bounded monotone solution of $-\Delta u = f(u)$ in $\R^3$ satisfies 
$$ 
\int_{B_R}|\nabla u|^2\di x = O(R^2) \qquad \text{as } \, R \ra +\infty 
$$
(see \cite{FSV, albertiambrosiocabre}). Hence, the conclusion follows by applying Theorem \ref{teo_DeGiorgi} with $ X = \partial/\partial x_3$.

\item[(iv)] By Remark \ref{rem_lira}, the flatness of $M$ is automatic in Corollary \ref{cor_DeGiorgi2D} from the sole assumptions $K \ge 0$ and $X$ Killing and nowhere vanishing.

\item[(v)] If $ m \ge 2$ and $M^m= N \times \R$ with the product metric $\metric = \metric_N + \di t^2$, then it is always possible to construct a solution of \eqref{equazu} which is monotone in the direction of the Killing vector field $\partial_t$ (proceed as in the proof of Lemma \ref{lem_easy}). Our main Theorem \ref{teo_DeGiorgi} states that the converse holds true if the manifold $M$ has non-negative Ricci curvature and it supports a De Giorgi-type conjecture.   
 
\end{itemize}

\end{remark}


%
%
%
\section*{Overdetermined boundary value problems}

In this section we study the case of overdetermined elliptic problems on open and connected sets with $C^3$ boundary. In the
situation considered here, the boundary term in \eqref{integrimpolim} may cause extra difficulties. Suprisingly, for solutions monotone in the direction of some Killing vector field, the boundary term indeed can be ruled out, as the next lemma reveals:
\begin{lemma}\label{lem_bordo}
Let $u$ be such that $u$ and $\partial_\nu u$ are constant on $\partial \Omega$ and $\partial_\nu u \neq 0$ on $\partial \Omega$.
Suppose that $w$ is of the form $w= \langle \nabla u, X\rangle$ in a neighbourhood of $\partial \Omega$, for some vector field $X$. Then
\begin{equation}\label{identity}
w\partial_\nu \left(\frac{|\nabla u|^2}{2}\right) - |\nabla u|^2 \partial_\nu w =  - |\nabla u|^3 \langle \nu, \nabla X(\nu) \rangle \qquad \text{on } \partial \Omega.
\end{equation} 
In particular, if $X$ satisfies $\langle \nabla X(\nu),\nu\rangle \ge 0$ on $\partial \Omega$, the boundary terms in \eqref{integrimpo} and \eqref{integrimpolim} are non-positive.
\end{lemma}
\begin{proof}
Let us define the constant $c= \partial_\nu u$ on $\partial\Omega$. Since $u$ is constant, $\nabla u = (\partial_\nu u)\nu = c\nu$, so $|\nabla u|^2 = c^2$ is constant on $\partial \Omega$. Therefore, its gradient has only normal component:
$$
\partial_\nu(|\nabla u|^2)\nu = \nabla(|\nabla u|^2) = 2 \nabla \di u (\nabla u, \cdot)^\sharp.
$$
It follows that, in our assumptions,
$$
\begin{array}{lcl}
\disp |\nabla u|^2 \partial_\nu w & = & \disp c^2 \partial_{\nabla u/c} w = c \nabla u \langle \nabla u, X\rangle \\[0.2cm]
& = & \disp c \big[\langle \nabla_{\nabla u}\nabla u, X\rangle + \langle \nabla u , \nabla_{\nabla u} X\rangle \big] = c\big[ \nabla \di u(\nabla u, X) + c^2 \langle \nu, \nabla X(\nu)\rangle\big] \\[0.2cm]
& = & \disp c \nabla \di u(\nabla u, X) + c^3 \langle \nu, \nabla X(\nu)\rangle = \frac{c}{2} \langle \nabla |\nabla u|^2, X \rangle + |\nabla u|^3 \langle \nu, \nabla X(\nu)\rangle \\[0.2cm]
& = & \disp \frac c2 \partial_\nu(|\nabla u|^2) \langle \nu, X \rangle + |\nabla u|^3 \langle \nu, \nabla X(\nu)\rangle = \frac{\partial_\nu(|\nabla u|^2)}{2} \langle \nabla u, X\rangle + |\nabla u|^3 \langle \nu, \nabla X(\nu)\rangle \\[0.2cm]
& = & \disp \frac{\partial_\nu(|\nabla u|^2)}{2} w + |\nabla u|^3 \langle \nu, \nabla X(\nu)\rangle,
\end{array}
$$
as claimed.
\end{proof}

\begin{remark}
\emph{Clearly, any Killing vector field fulfills the requirement $\langle \nabla X(\nu), \nu \rangle \ge 0$, but the class is much more general. For instance, $\langle \nabla X(\nu), \nu \rangle \ge 0$ is met whenever $X$ solves 
$$
L_X \metric \ge 0 \qquad \text{as a quadratic form.}
$$
Examples of such $X$ also include positively conformal vector fields, that is, fields satisfying $L_X\metric = \eta \metric$ for a non-negative $\eta \in C^\infty(M)$, and gradients of convex functions $X= \nabla \psi$, being $L_{\nabla \psi} \metric = 2 \nabla \di \psi$. 
} 
\end{remark}
%

%

%

%









The above Lemma is the key to prove Theorem \ref{teo_mainbordo}. 

\begin{proof}[\textbf{Proof of Theorem \ref{teo_mainbordo}}]
In our assumptions, by Lemma \ref{lem_bordo} either $w=\langle\nabla u,X\rangle$ or $w=-\langle\nabla u,X\rangle$ is a positive solution on $\Delta w + f'(u)w=0$ on $\Omega$. Up to changing the sign of $X$, we can suppose that $w= \langle \nabla u, X\rangle > 0$ on $\Omega$. In particular, $X$ is nowhere vanishing on $\Omega$. We are going to show that condition \eqref{assunz} is satisfied, namely, that 
\begin{equation}
\liminf_{\eps\rightarrow 0^+}\int_\Omega \phi^2\langle \nabla|\nabla u|^2, \nabla\left( \frac{w}{w+ \eps}\right) \rangle \, \di x \ge 0.
\end{equation}
By a density argument, this will be accomplished once we prove that
\begin{equation}\label{condition}
\liminf_{\eps\rightarrow 0^+}\int_K \langle \nabla|\nabla u|^2, \nabla\left( \frac{w}{w+ \eps}\right) \rangle \, \di x \ge 0 \qquad \forall \, K\Subset \overline \Omega.
\end{equation}
We first claim that there exists a constant $C=C(K,m, \|u\|_{C^3(K)}) >0$ such that 
\begin{equation}\label{P.P2}
\big|\langle \nabla |\nabla u|^2, \nabla w \rangle \big|\le C |w| \qquad \text{on } K \cap\partial \Omega.
\end{equation}
First we observe that, since $u$ is constant on $\partial \Omega$, $\nabla u$ has only normal component, thus
\begin{equation}\label{F6}
|w| = \big|\langle \nabla u , X \rangle \big| = |\nabla u|\, | \langle \nu,X\rangle |.
\end{equation}
From the further property that $|\nabla u|^2= c^2$ is constant along $\partial\Omega$, we deduce that $\nabla |\nabla u|^2$ is parallel to $\nu$ and therefore, by Kato inequality,
\begin{equation}\label{F1}
\big|\langle \nabla |\nabla u|^2 , \nabla w \rangle \big|= \big|\nabla |\nabla u|^2 \big|\, |\partial_\nu w| \le 2|\nabla u|\, |\nabla \di  u| \,|\partial_\nu w| \end{equation}
on $\partial\Omega$. Using the fact that $X$ is a Killing vector field and $\nabla u = c\nu$ on $\partial \Omega$, the following chain of equalities is true:
\begin{equation}\label{F4}
\begin{array}{lcl}
\partial_\nu w & = & \disp \nu (\langle \nabla u, X\rangle) = \nabla \di u(\nu,X) + \langle \nabla u, \nabla_\nu X \rangle \\[0.2cm]
& = & \nabla \di u (\nu,X) + c \langle \nu, \nabla_\nu X \rangle = \nabla \di u(\nu,X).
\end{array}
\end{equation}
Now, we use that $\partial_\nu u$ is constant on $\partial\Omega$, whence $\nabla (\partial_\nu u)$ is also parallel to $\nu$:
\begin{equation}\label{F2}
\pm \big|\nabla (\partial_\nu u)\big|\nu = \nabla (\partial_\nu u) = \nabla \big(\langle \nabla u, \nu\rangle\big) = \nabla \di u(\nu, \cdot)^\sharp + \langle \nabla u, \nabla\nu \rangle.  
\end{equation}
on $\partial\Omega$, where $\nabla \nu$ is the $(1,1)$ version of the second fundamental form of $\partial \Omega$, that is, the (opposite of the) Weingarten transformation. Taking the inner product with $X$ we deduce that, on $\partial \Omega$,
\begin{equation}
\begin{array}{lcl}
\disp \pm \big|\nabla (\partial_\nu u)\big| \langle \nu, X\rangle & = & \disp \nabla \di u(\nu, X) + \langle \nabla u, \nabla_X\nu \rangle = \nabla \di u(\nu, X) + c \langle \nu, \nabla_X\nu \rangle \\[0.2cm] 
& = &  \nabla \di u(\nu, X) + \frac c2 X(|\nu|^2) =  \nabla \di u(\nu, X), 
\end{array}
\end{equation}
whence combining with \eqref{F4} we conclude
\begin{equation}\label{quasibuona}
|\partial_\nu w| = \big|\nabla \di u(\nu,X)\big| = \big|\nabla(\partial_\nu u)\big|\, |\langle\nu,X\rangle |.
\end{equation}
Inserting the equalities \eqref{F6} and \eqref{quasibuona} into \eqref{F1} we deduce
$$
\begin{array}{lcl}
\disp \big|\langle \nabla |\nabla u|^2 , \nabla w \rangle \big| \le 2|\nabla u|\, |\nabla \di  u| \,|\partial_\nu w| & \le & \disp  2|\nabla u|\, |\nabla \di  u| \,|\nabla (\partial_\nu u)| \, |\langle \nu,X\rangle| \\[0.2cm]
& = & \disp 2 |\nabla \di  u| \,|\nabla (\partial_\nu u)|\, |w|.
\end{array}
$$
Since $u \in C^{3}(\overline \Omega)$, the terms $\nabla \di u$ and $\nabla(\partial_\nu u)$ are bounded on $K$, and the claimed inequality \eqref{P.P2} is proved.\\
Our next task is to extend the bound in \eqref{P.P2} to a whole neighbourhood of $\partial\Omega$. More precisely, we claim that there exist $C>0$, possibly depending on $K$, $f$, $u$ and $\partial\Omega$, such that
\begin{equation}\label{scope}
\big|\langle \nabla |\nabla u|^2, \nabla w\rangle\big|\le C|w|\qquad \forall \, x \in \Omega \cap K. 
\end{equation}
To prove this, we notice that it is enough to prove the bound in a neighbourhood of $K \cap \partial\Omega$. By the compactness of $K \cap \partial\Omega$, it is enough to work locally around any $x_0 \in \partial\Omega$. Towards this aim we note that, since $\partial \Omega$ is $C^3$, for any $x_0 \in \partial\Omega$ Fermi coordinates $(T,\Psi)$ can be defined in a collar $T \Subset \Omega$ of $x_0$:
$$
\Psi \, : \, T \longrightarrow [0,\delta) \times U \ \subseteq \ \R_0^+ \times \partial \Omega, \qquad \Psi(x) = (t, \pi(x)),
$$
where $U$ is open in $\partial \Omega$ and contains $x_0$. In particular, $\pi(x) \in \partial \Omega$ is the unique point of $\partial \Omega$ realizing $\dist(x,\partial\Omega)$, and the smooth coordinate $t \in [0,\delta)$ satisfies 
$$
t(x) = \dist(x, \partial \Omega)= \dist(x, \pi(x)).
$$
Again since $\partial \Omega$ is smooth enough, up to shrinking further $T$ there exists a bounded domain $D_0\Subset \Omega$, of class $C^3$ and containing $T$, that satisfies 
$$
t(x) = \dist(x, \partial \Omega) = \dist(x, \partial D_0) \qquad \forall \, x \in T.
$$
In the chart $\Psi$, the function $w \in C^2(\overline T)$ satisfies a linear elliptic equation, the expression in chart of $\Delta w + f'(u)w=0$, to which the Hopf-type Lemma 1 of \cite{walter} can be applied to deduce
\begin{equation}\label{first}
w(x) \ge C \dist (x, \partial D_0)= Ct(x) \qquad \forall \, x \in T,
\end{equation}
for some $C>0$. 
Next, since $u \in C^{3}(\overline T)$, the function 
$$
g(x) = \big|\langle \nabla |\nabla u|^2, \nabla w \rangle \big| \in \lip(\overline T),
$$
whence
\begin{equation}\label{second}
\big| g(\pi(x)) - g(x) \big| \le C \dist (\pi(x),x) = Ct(x)
\end{equation}
All in all, combining \eqref{first} and \eqref{second}, and using also \eqref{P.P2} we obtain:
\begin{equation}
\disp \frac{g(x)}{|w(x)|} \le  \disp \frac{|g(x)-g(\pi(x))|+|g(\pi(x))|}{|w(x)|} \le \frac{|g(x)-g(\pi(x))|+C |w(x)|}{|w(x)|} \le C,
\end{equation}

%
for a suitable $C>0$. This completes the proof of \eqref{scope}.
Now we observe that the integrand in \eqref{condition} may be written as
$$ 
\frac{\eps}{(\eps+w)^2}\langle \nabla|\nabla u|^2, \nabla w \rangle =: \psi_\eps.
$$
Notice that $\psi_\eps$ is well-defined in $\Omega$ since $w>0$, and
$$ 
\lim_{\eps\rightarrow 0^+}\psi_\eps(x)=0,
$$
for each $x\in\Omega$. Moreover, by \eqref{scope}, on $\overline T$
$$ 
|\psi_\eps(x)| \le\frac{(\eps+w)}{(\eps+w) w} \big|\langle \nabla|\nabla u|^2, \nabla w \rangle \big|\le C.
$$
Then, \eqref{condition} with $K=T$ follows from Lebesgue convergence theorem.\\
Applying Proposition \ref{epsazero} with the aid of Lemma \ref{lem_bordo}, the boundary term in \eqref{integrimpolim} vanishes since $X$ is Killing, and we get 
\begin{equation}\label{integrimpolim2}
\begin{array}{lcl}
\disp \int_\Omega \left[ |\nabla \di u|^2 + \Ricc(\nabla u, \nabla u)- \big|\nabla |\nabla u|\big|^2 \right]\phi^2 \di x 
\disp + \liminf_{\eps \ra 0^+} \int_\Omega (w+\eps)^2 \left|\nabla\left(\frac{\phi|\nabla u|}{w+\eps}\right)\right|^2\di x  \\[0.5cm]
\disp \le \int_\Omega |\nabla \phi|^2|\nabla u|^2 \di x. 
\end{array}
\end{equation}
Hereafter, we can proceed in a way analogous to that in Theorem \ref{teo_main}. In particular, the use of appropriate cutoff functions $\{\phi_\alpha\}$ satisfying \eqref{claim}, and the assumption $\Ricc \ge 0$, imply 
\begin{equation}\label{equakato}
|\nabla \di u|^2 = \big|\nabla |\nabla u|\big|^2, \qquad \Ricc(\nabla u, \nabla u) = 0 \qquad \text{on }\Omega, 
\end{equation}
thus inserting into \eqref{integrimpolim2} we obtain
\begin{equation}\label{orafacile}
\liminf_{\eps \ra 0^+} \int_\Omega (w+\eps)^2 \left|\nabla\left(\frac{\phi|\nabla u|}{w+\eps}\right)\right|^2\di x  
\disp \le \int_\Omega |\nabla \phi|^2|\nabla u|^2 \di x. 
\end{equation}
For every small $\delta >0$, we define $\Omega_\delta = \{x \in \Omega : \dist (x,\partial \Omega) > \delta\}$. By the positivity of the integrand, and since away from $\partial \Omega$ the function $w$ is locally uniformly bounded away from zero,
$$
\begin{array}{lcl}
\disp \liminf_{\eps \ra 0^+} \int_\Omega (w+\eps)^2 \left|\nabla\left(\frac{\phi|\nabla u|}{w+\eps}\right)\right|^2\di x  & \ge & \disp  
\lim_{\eps \ra 0^+} \int_{\Omega_\delta} (w+\eps)^2 \left|\nabla\left(\frac{\phi|\nabla u|}{w+\eps}\right)\right|^2\di x \\[0.4cm]
& = & \disp \int_{\Omega_\delta} w^2 \left|\nabla\left(\frac{\phi|\nabla u|}{w}\right)\right|^2\di x.
\end{array}
$$
Letting $\delta \ra 0$ we thus get 
$$
\liminf_{\eps \ra 0^+} \int_\Omega (w+\eps)^2 \left|\nabla\left(\frac{\phi|\nabla u|}{w+\eps}\right)\right|^2\di x  \ge 
\int_\Omega w^2 \left|\nabla\left(\frac{\phi|\nabla u|}{w}\right)\right|^2\di x.  
$$
In particular, by \eqref{orafacile} the RHS of the above inequality is finite and
$$
\int_{\Omega} w^2 \left|\nabla\left(\frac{\phi|\nabla u|}{w}\right)\right|^2\di x \le \int_\Omega |\nabla \phi|^2|\nabla u|^2 \di x.
$$
An application of Young type inequality \eqref{young} transforms the above inequality into
$$
(1-\delta) \int_{\Omega} w^2 \left|\nabla\left(\frac{|\nabla u|}{w}\right)\right|^2\di x \le \frac 1 \delta \int_\Omega |\nabla \phi|^2|\nabla u|^2 \di x,
$$
for each $\delta \in (0,1)$. Consequently, choosing again the appropriate cut-offs $\{\phi_\alpha\}$ satisfying \eqref{claim} as in Theorem \ref{teo_main}, we also get $|\nabla u| = cw$ for some constant $c \ge 0$. Since $u$ is non-constant, $c>0$. The topological part of the splitting needs some extra care. We shall divide into two cases, according to the sign of the constant $\partial _\nu u $ on $\partial \Omega$. Since the discussions are specular, we just consider the case when $\partial_\nu u$ is positive on $\partial \Omega$. Denote with $N \subseteq \Omega$ any level set of $u$, and with $\Phi_t$ the flow of of $\nu = \nabla u/|\nabla u|$ on $\Omega$. Observe that, for $x \in \partial \Omega$, the fact that $u \circ \Phi_t$ is strictly increasing implies that $\Phi_t(x) \in \Omega$ for each $t \in \R^+$. From the Sternberg-Zumbrun identity in Proposition \ref{sternzun}, $N$ is totally geodesic, $|\nabla u|$ is constant (and non-zero) on $
 N$ and the only non-vanishing component of $\nabla \di u$ is that corresponding to the pair $(\nu,\nu)$. Therefore, integral curves of $\nu$ are geodesics. Write $|\nabla u| = \beta(u)$, for some continuous $\beta$. We claim that, for each $x \in \Omega$, $\Phi_t(x)$ touches $\partial \Omega$ at a finite, negative time $t_0(x)$. Indeed, consider the rescaled flow $\Psi_s$ of the vector field $Y = \nabla u/|\nabla u|^2$. Clearly, $\Phi(t,x)= \Psi(s(t),x)$, where
$$
s(t) = s(0) + \int_0^t |\nabla u|(\Phi_\tau(x))\di \tau = s(0)+ \int_0^t \beta\big(u\circ \Phi_\tau(x)\big) \di \tau
$$
is a locally Lipschitz bijection with inverse $t(s)$. From $u (\Psi_s (x)) = u(x) + s$ and from $\partial_\nu u >0$, we deduce that $\Psi_s(x)$ touches $\partial \Omega$ at a finite, negative value $s_0(x)$. Now, since $\Phi_t(x)$ is a geodesic, and geodesics are divergent as $t\ra -\infty$, then necessarily the correspondent value $t_0(x) = t(s_0(x))$ is finite. Consequently, the flow of $\nu$ starting from $\partial \Omega$ covers the whole $\Omega$. Having fixed a connected component $\Sigma$ of $\partial \Omega$, proceeding as in the proof of Theorem \ref{teo_main} it can be shown that $\Phi : \Sigma \times \R^+ \ra \Omega$ is a $C^3$ diffeomorphism. Thus, $\Sigma \equiv \partial \Omega$ and we have the desired topological splitting. The proof that each $\Phi_t$ is an isometry is identical to the boundaryless case. It thus follows, via a simple approximation, that $\partial \Omega$ is totally geodesic and isometric to any other level set of $\Omega$, and thus $\Omega$ splits as a Riemannian product $\partial \Omega\times \R^+$. Setting $u(x,t)=y(t)$, $y$ solves 
$$
y'(t) = |\nabla u|(x,t) >0, \qquad y''(t)=-f\big(y(t)\big).
$$
As regards the volume estimate for $\partial \Omega$, it follows exactly along the same lines as those yielding \eqref{volumepar}: 
$$ \Big(\int_{0}^R |y'(t)|^2\di t \, \Big) \vol(B_R^{\partial \Omega}) \le \int_{(0,R] \times B_R^{\partial \Omega}} |y'(t)|^2\di t \, \di x^{\partial \Omega}  \le \int_{B_{R\sqrt 2} \cap \Omega} |\nabla u|^2 \di x = o (R^2 \log R) 
$$
as $R \ra +\infty$, according to $(ii)$. Lastly, we address the mutual position of $X$ and $\partial_t = \nu$. From the identity $|\nabla u|= cw= c\langle \nabla u, X\rangle$ we deduce that 
$$
\langle \partial_t, X\rangle = \frac{1}{|\nabla u|} \langle \nabla u, X\rangle = \frac{1}{c}
$$
is constant on $M$. Consequently, the projected vector field
$$
X^\perp = X - \langle X,\partial_t \rangle \partial_t
$$
is still a Killing field, since so are $X$ and $\partial_t$. This concludes the proof. The case $\partial_\nu u <0$ on $\partial \Omega$ can be dealt with analogously, by considering the flow of $\nu = - \nabla u/|\nabla u|$.
\end{proof}
Clearly, in the above theorem a key role is played by the monotonicity condition $\langle \nabla u, X\rangle >0$, for some Killing vector field $X$. As remarked in the Introduction, this condition is automatically satisfied for globally Lipschitz epigraphs $\Omega \subseteq \R^m$, and for $f \in \lip(\R)$ satisfying some mild assumptions, thanks to the following remarkable result by H. Berestycki, L. Caffarelli and L. Nirenberg in \cite{bercaffnire}:
\begin{theorem}[\cite{bercaffnire}, Theorem 1.1]
Let $\Omega \subseteq \R^m$ be an open subset that can be written as the epigraph of a globally Lipschitz function $\varphi$ on $\R^{m-1}$, that is, 
$$
\Omega = \big\{ (x', x_m) \in \R^m = \R^{m-1} \times \R \ : \ x_m > \varphi(x')\big\}.
$$
Let $f \in \lip(\R)$ satisfy the requirements
$$
\left\{\begin{array}{l}
\disp f >0 \quad \text{on } (0, \lambda), \qquad f \le 0 \quad \text{on } (\lambda, +\infty), \\[0.3cm]
\disp f(s) \ge \delta_0s \quad \text{for } s \in (0, s_0), \\[0.2cm]
f \text{ is non-increasing on } [\lambda -s_0, \lambda],
\end{array}\right.
$$
for some positive $\lambda, \delta_0, s_0$. Let $u \in C^2(\Omega)\cap C^0(\overline \Omega)$ be a bounded, positive solution of 
\begin{equation}\label{problediri}
\left\{ \begin{array}{l}
-\Delta u = f(u) \qquad \text{on } \Omega, \\[0.2cm]
u>0 \quad \text{on } \Omega, \qquad u=0 \quad \text{on } \partial \Omega.
\end{array} \right.
\end{equation}
Then, $u$ is monotone in the $x_m$-direction, that is, $\partial u/\partial x_m >0$ on $\Omega$.
\end{theorem}
The proof of this result relies on some techniques which are tightly related to the peculiarities of Euclidean space as a Riemannian manifold. It would be therefore very interesting to investigate the following
\begin{quote}
\textbf{problem: } determine reasonable assumptions on the manifold $(M, \metric)$ and on  $\Omega,f$ which ensure that every bounded, sufficiently smooth solution $u$ of \eqref{problediri}, or at least of \eqref{pser}, is monotone in the direction of a Killing vector field $X$.
\end{quote}
In the next section, we prove some preliminary results addressed to the above problem. In doing so, we obtain an improvement of Theorem \ref{teo_mainbordo} in the dimensional case $m=3$.
\section*{Further qualitative properties of solutions, and the monotonicity condition}
This last section is devoted to move some first steps towards a proof of the monotonicity condition in a manifold setting. In doing so, we extend results in \cite{bercaffnire}, \cite{berenire} to Riemannian manifolds satisfying $\Ricc \ge -(m-1)H^2 \metric$, for some $H \ge 0$. Although the proofs below are in the same spirit as those in \cite{bercaffnire} and \cite{berenire}, in order to deal with the lack of symmetry of a general $M$ we shall introduce some different arguments that may have independent interest. In particular, we mention Proposition \ref{prop_elegante} for its generality. Combining the results of this section will leads us to a proof of Theorem \ref{teo_speciale}. Hereafter, we shall restrict ourselves to a class of nonlinearities $f$ satisfying the following general assumptions:
$$
\left\{\begin{array}{l}
\disp f >0 \quad \text{on } (0, \lambda), \qquad f(\lambda)=0, \qquad f<0 \quad \text{on } (\lambda, \lambda + s_0), \\[0.3cm]
\disp f(s) \ge \delta_0s \quad \text{for } s \in (0, s_0),
\end{array}\right.
$$
for some $\lambda>0$ and some small $\delta_0,s_0>0$. Let $\Omega \subseteq M$ be an open, connected subset with possibly noncompact closure, and let $u \in C^2(\Omega)\cap C^0(\overline \Omega)$, $u>0$ on $\Omega$ solve
$$
-\Delta u = f(u) \qquad \text{on } \Omega.
$$
For $R_0>0$, set
$$
\Omega_{R_0} = \big\{ x \in \Omega \, : \, \dist(x, \partial \Omega) > R_0 \big\}, \qquad  \Omega^{R_0} = \big\{ x \in \Omega \, : \, \dist(x, \partial \Omega) < R_0 \big\}.
$$
Moreover, for notational convenience, for $y \in M$ define $r_y(x) = \dist(y,x)$. 
\begin{remark}
\emph{We observe that, even when $\Omega$ is connected, $\Omega_{R_0}$ may have infinitely many connected components. By a compactness argument, however, such a number is always finite if $\Omega$ is relatively compact.
}
\end{remark}
The first lemma ensures that, for suitable $f$, $u$ is bounded from below by some positive constant on each connected component of $\Omega_{R_0}$. The strategy of the proof is somehow close to the spirit of the sliding method, although this latter cannot be applied due to the lack of a group of isometries acting transitively on $M$.

\begin{lemma}\label{lem_infnonzero}
Let $(M^m, \metric)$ be a complete Riemannian manifold such that $\Ricc \ge -(m-1)H^2 \metric$, for some $H \ge 0$. Suppose that $f \in C^1(\R)$ satifies
\begin{equation}\label{ipofzero}
f(s) \ge \left( \delta_0 + \frac{(m-1)^2H^2}{4}\right) s \qquad \text{ for } s \in (0, s_0), 
\end{equation}
for some positive, small $\delta_0, s_0$. Let $u \in C^2(\Omega)$ be a positive solution of $-\Delta u = f(u)$ on $\Omega$. Then, there exists $R_0= R_0(m,H,\delta_0)>0$ such that the following holds: for each connected component $V_j$ of $\Omega_{R_0}$, there exists $\eps_j = \eps_j(\delta_0, H,m, V_j) >0$ such that
\begin{equation}\label{stimasotto}
u(x) \ge \eps_j \qquad \text{if } x \in V_j.
\end{equation}
In particular, if $\Omega_{R_0}$ has only finitely many connected components,
$$
\inf_{\Omega_{R_0}}u > 0
$$
\end{lemma}

\begin{proof}
Let $M_H$ be a space form of constant sectional curvature $-H^2 \le 0$ and dimension $m$. In other words, $M_H= \R^m$ for $H=0$, and $M_H$ is the hyperbolic space of curvature $-H^2$ if $H>0$. Let $o \in M_H$. For $R>0$, denote with $\lambda_1(\bh_R)$ the first Dirichlet eigenvalue of $-\Delta$ on the geodesic ball $\bh_R=B_R(o) \subseteq M_H$. By a standard result (combine for instance \cite{mckean} and \cite{brooks}), the bottom of the spectrum of $-\Delta$ on $M_H$, $\lambda_1(M_H)$, is given by
$$
\lambda_1(M_H) = \lim_{R \ra +\infty} \lambda_1(\bh_R) = \frac{(m-1)^2H^2}{4}. 
$$
Therefore, by \eqref{ipofzero} we can choose $R_0 = R_0(\delta_0, H,m)$ such that, for $R \ge R_0 /2$, 
\begin{equation}\label{ipoinff}
\lambda_1(\bh_R)s < f(s) \qquad \text{for } s \in (0, s_0]. 
\end{equation}
Let $\sn(r)$ be a solution of
$$
\left\{\begin{array}{l}
\sn''(r) - H^2 \sn(r) = 0 \qquad \text{on } \R^+ \\[0.2cm]
\sn(0)=0, \qquad \sn'(0)=1
\end{array}\right.
$$
and set $\cn(r) = \sn'(r)$. Moreover, let $z$ be a first eigenfunction of $\bh_R$. Then, via a symmetrization argument and since the space of first eigenfunctions has dimension $1$, $z$ is radial and (up to normalization) solves
$$
\left\{\begin{array}{l}
z''(r) + (m-1)\dfrac{\cn(r)}{\sn(r)}z'(r) + \lambda_1(\bh_R)z(r) = 0 \qquad \text{on } (0,R), \\[0.3cm]
z(0)=1, \quad z'(0)=0, \quad z(R)=0, \quad z(r)>0 \, \text{ on } [0,R). 
\end{array}\right.
$$
A first integration shows that $z' <0$ on $(0,R)$, so $z \le 1$. From assumption $\Ricc \ge -(m-1)H^2 \metric$ and the Laplacian comparison theorem (see for instance \cite{petersen}, Ch. 9 or \cite{prs}, Section 2), we deduce that
$$
\Delta r_y (x) \le \frac{\cn}{\sn}(r_y(x))
$$
pointwise outside the cut-locus of $y$ and weakly on the whole $M$. Therefore, for every $y \in \Omega_R$, the function $\varphi_y \, : \, M \ra \R$ defined as 
$$
\varphi_y(x) = \left\{ \begin{array}{ll}
z(r_y(x)) & \quad \text{if } x \in B_R(y) \\[0.1cm]
0 & \quad \text{otherwise}
\end{array}\right.
$$
is a Lipschitz, weak solution of
\begin{equation}\label{laphiy}
\left\{\begin{array}{l}
\Delta \varphi_y + \lambda_1(\bh_R)\varphi_y \ge 0 \qquad \text{on } B_R(y) \\[0.2cm]
0 < \varphi_y \le 1 \, \text{ on } B_R(y), \quad \varphi_y = 0 \, \text{ on } \partial B_R(y), \quad \varphi_y(y)=1.
\end{array}\right.
\end{equation}
Fix any $R \in (R_0/2, R_0)$. Note that, in this way, for each $y \in \Omega_{R_0}$ it holds $B_R(y) \Subset \Omega$, and \eqref{ipoinff} is met. Let $\{V_j\}$ be the connected components of $\Omega_{R_0}$ (possibly, countably many). For each $j$, choose $y_j \in V_j$ and $\eps_j \in (0, s_0)$ sufficiently small that
$$
\eps_j\varphi_{y_j}(x) < u(x) \qquad \text{for every } x \in B_R(y_j) \Subset \Omega.
$$
This is possible since $u>0$ on $\overline{B_R(y_j)}$. Let $y \in V_j$. From $\varphi_y \le 1$, $\eps_j \varphi_y \le s_0$, thus
\begin{equation}\label{interesting}
f(\eps_j \varphi_y)  > \lambda_1(\bh_R) \eps_j \varphi_y \qquad \text{on } B_R(y).
\end{equation}
We are going to show that, for each $y \in V_j$, $u(y) \ge \eps_j$. Towards this aim, let $\gamma :[0,l] \ra V_j$ be a unit speed curve joining $y_j$ and $y$, in such a way that $\gamma(0)= y_j$. Define
$$
w_t(x) = u(x) - \eps_j \varphi_{y(t)}(x) = u(x) - \eps_j z\big( \dist(y(t), x)\big).
$$
Then, the curve $w : t \in [0,l] \ra w_t \in C^0(\overline\Omega)$, where $C^0(\overline\Omega)$ is endowed with the topology of uniform convergence on compacta, is continuous. Indeed, by the triangle inequality and since $\gamma$ has speed $1$ we have
$$
\begin{array}{lcl}
\disp \|w(t) - w(s) \|_{L^\infty(\Omega)} &=& \eps_j\disp \left\|z\big( \dist(\cdot, y(t))\big) - z\big( \dist(\cdot, y(s))\big) \right\|_{L^\infty(\Omega)} \\[0.2cm]
& \le & s_0\lip(z) \left\|\dist(\cdot, y(t)) - \dist(\cdot, y(s)) \right\|_{L^\infty(\Omega)} \\[0.2cm]
& \le & \disp s_0\lip(z) \dist (y(t), y(s)) \le \eps_0\lip(z)|t-s|.
\end{array}
$$
It follows that the set
$$
T = \Big\{ t \in [0,l] \ : \ i(t) = \inf_{B_R(y(t))}w_t >0\Big\}
$$
is open on $[0,l]$ and contains $t=0$. We stress that, for each $t \in [0,l]$, by construction $B_R(y(t)) \Subset \Omega$. We claim that $T=[0,l]$. If not, there is a first point $a \le l$ such that $i(t)>0$ for $t \in [0,a)$ and $i(a)=0$. Since $w_a >0$ on $\partial B_R(y(a))$ by construction, the minimum of $w_a$ is attained on some $x_0 \in B_R(y(a))$. Now, by \eqref{laphiy} and \eqref{interesting}, $w_a$ solves weakly
$$
\Delta w_a = -f(u) + \eps_j \lambda_1(\bh_R) \varphi_{y(a)} < - f(u) + f(\eps_j \varphi_{y(a)}) = c(x) w_a,
$$
where as usual 
$$
c(x) = \frac{f(\eps_j \varphi_{y(a)}(x))-f(u(x))}{u(x) -\eps_j \varphi_{y(a)}(x) }.
$$
Now, from $w_a \ge 0$ and $w_a(x_0)=0$, by the local Harnack inequality for Lipschitz weak solutions (see \cite{gilbargtrudinger}), $w_a \equiv 0$, contradicting the fact that $w_a>0$ on $\partial B_R(y(a))$. This proves the claim.\\
Now, from $w_l >0$ on $\overline{B_R(y)}$, in particular
$$
0 < w_l(y) = u(y) - \eps_j \varphi_y(y) = u(y) - \eps_j, 
$$
proving the desired \eqref{stimasotto}.
\end{proof}
In the second Lemma, we specify the asymptotic profile of the solution as $\dist (x,\partial \Omega) \ra +\infty$. First, we shall need some notation. Let $R>0$, and consider the (radial) solution $v_R$ of
$$
\left\{\begin{array}{l}
\Delta v_R = -1 \qquad \text{on } \bh_R \subseteq M_H, \\[0.2cm]
v_R = 0 \qquad \text{on } \partial \bh_R,
\end{array}\right.
$$
where $\bh_R = B_R(o) \subseteq M_H$ as in the previous lemma. Since $v_R$ is radial, integrating the correspondent ODE we see that 
$$
v_R(r) = \int_r^R \frac{1}{\sn(t)^{m-1}} \left[\int_0^t \sn(s)^{m-1}\di s\right] \di t.
$$
Denote with $C_H(R)= \|v_R\|_{L^\infty([0,R])}= v_R(0)$, and  observe that 
\begin{equation}\label{CHR}
C_H(R) \downarrow 0^+ \quad \text{as } \, R \ra 0^+, \qquad C_H(R) \uparrow +\infty \quad \text{as } \, R \uparrow +\infty.
\end{equation}
Let $R_0, \{V_j\}$ and $\eps_j$ be as in the previous lemma, and for $y \in V_j$ set
\begin{equation}\label{defdeltaora}
\delta_j(y) = \min\big\{ f(s) \, : \, s \in [\eps_j, u(y)] \big\}.
\end{equation} 

\begin{lemma}\label{lem_approaching} 
In the assumptions of the previous lemma, suppose further that $u$ is bounded above, and that 
\begin{equation}\label{fmagcompo}
f>0 \qquad \text{on } (0, \|u\|_{L^\infty}).
\end{equation}
Then, for every $y \in V_j \subseteq \Omega_{R_0}$,
\begin{equation}\label{approaching}
\delta_j(y) C_H\big( [ \dist( y, \partial \Omega) - R_0] \big) \le \|u\|_{L^\infty}.
\end{equation}

\end{lemma}

\begin{proof}
Under assumption \eqref{fmagcompo}, $\delta_j(y) \ge 0$ for each $y \in V_j$ and each $j$. Suppose by contradiction that there exists $y \in V_j$ such that
$$
\delta_j(y) C_H\big( [ \dist( y, \partial \Omega) - R_0] \big) > \|u\|_{L^\infty}, 
$$
and let $R<\dist(y, \partial \Omega) - R_0$ be such that 
\begin{equation}\label{contradiction}
\delta_j(y) C_H(R) >\|u\|_{L^\infty}.
\end{equation}
Note that $\delta_j(y)>0$ and that, with such a choice of $R$, $B_R(y) \Subset V_j$. By the positivity of $\delta_j(y)$ and since $u(y)>0$, $\Delta u(y) = -f(u(y)) <0$. Thus, arbitrarily close to $y$ we can find a point $\bar y \in V_j$ such that $u(\bar y) < u(y)$, and we can choose $\bar y$ in order to satisfy the further relation $y \in B_R(\bar y) \Subset V_j$. Define 
$$
h(x) = \delta_j(y) v_R\big( r_{\bar y}(x) \big).
$$
Since $v_R' < 0$, by the Laplacian comparison theorem it holds
\begin{equation}\label{lah}
\left\{ \begin{array}{l}
\Delta h \ge - \delta_j(y) \qquad \text{weakly on } B_R(\bar y), \\[0.2cm]
h = 0 \qquad \text{on } \partial B_R(\bar y).
\end{array}\right.
\end{equation}
Note that $\|h\|_{L^\infty(B_R(\bar y))} = \delta_j(y) \|v_R\|_{L^\infty([0,R])} = \delta_j(y) C_H(R)$, and that the norm of $h$ is attained at $\bar y$. For $\tau>0$, define on $B_R(\bar y)$
$$
w(x) = \tau h(x) - u(x).
$$
If $\tau$ is small enough, then $w<0$. Choose $\tau$ to be the first value for which $\tau h$ touches $u$ from below. Hence, $w \le 0$ and there exists $x_0$ such that $w(x_0)=0$. From $h=0$ on $\partial B_R(\bar y)$, we deduce that $x_0 \in B_R(\bar y)$ is an interior point. From our choice of $h$ and $\bar y$, 
$$
\tau \delta_j(y) C_H(R) = \tau h(\bar y) = w(\bar y) + u(\bar y) \le u(\bar y) < u(y) \le \|u\|_{L^\infty(\Omega)}. 
$$
By assumption \eqref{contradiction}, we deduce that necessarily $\tau <1$. Now, from
$$
u(x_0) = \tau h(x_0)  \le \tau h(\bar y) = w(\bar y) + u(\bar y) \le u(\bar y) < u(y),
$$
there exists a small neighbourhood $U\subset B_R(\bar y)$ of $x_0$ such that $u_{|U} < u(y)$. But then, on $U$,
\begin{equation}\label{lauanco}
\Delta u = -f(u) \le - \min_U (f \circ u) \le - \min_{t \in [\eps_j,u(y)]} f(t) = -\delta_j(y).
\end{equation}
Finally, combining \eqref{lah} and \eqref{lauanco}, from $\tau<1$ $w$ satisfies
$$
\left\{ \begin{array}{l}
\Delta w = \tau \Delta h - \Delta u \ge -\tau \delta_j(y) + \delta_j(y) >0 \qquad \text{weakly on } V, \\[0.2cm]
w \le 0 \quad \text{on } V, \qquad w(x_0)=0,
\end{array}\right.
$$
which contradicts the maximum principle and proved the desired \eqref{approaching}. 
\end{proof}
Putting together the two theorems leads to the proof of Proposition \ref{prop_elegante}.
\begin{proof}[\textbf{Proof of Proposition \ref{prop_elegante}}]
In our assumptions, we can modify the function $u$ in a tiny neighbourhood $T\subseteq \Omega$ of $\partial \Omega$ to produce a function $\bar u \in C^2(M)$ such that $\bar u = u$ on $\Omega \backslash T$, $\sup_{T \cup (M\backslash \Omega)} \bar u < \|u\|_{L^\infty}$. For instance, choose $\eps>0$ be such that $\sup_{\partial \Omega} u + \eps < u^*$, and let $\psi \in C^\infty(\R)$ be such that 
$$
0 \le \psi \le 1, \quad \psi=0 \, \text{ on } \, \left(-\infty, \sup_{\partial \Omega}u + \frac \eps 2\right), \quad \psi=1 \, \text{ on } \, \left( 
\sup_{\partial\Omega} u +
\eps, +\infty\right).
$$
Then, $\bar u(x)= \psi(u(x))$ (extended with zero on $M\backslash \Omega$) meets our requirements. Denote with $u^* = \sup_M \bar u = \sup_\Omega u >0$. In our assumptions on the Ricci tensor, the strong maximum principle at infinity holds on $M$ (see \cite{prsmemoirs} and Appendix 1 below), thus we can find a sequence $\{x_k\}$ such that 
$$
\bar u(x_k) > u^* - \frac  1k, \qquad \frac 1 k \ge \Delta \bar u(x_k)  = -f(\bar u(x_k)). 
$$For $k$ large enough, by the first condition $x_k \in \Omega\backslash \overline T$, thus $u = \bar u$ around $x_k$. Letting $k \ra +\infty$ we deduce that $f(u^*) \ge 0$. Since $f(u(\Omega))$ is connected in $\R^+_0$ and contains $0$, in our assumptions $u^* \le \lambda$. Applying Lemmata  \ref{lem_infnonzero} and  \ref{lem_approaching} we infer the existence of a large $R_0$ such that, for each connected component $V_j$ of $\Omega_{R_0}$ and for each $y \in V_j$, 
$$
\delta_j(y) C_H\big(\dist(y, \partial \Omega) - R_0\big) \le \lambda,
$$
where $\delta_j(y)$ is defined in \eqref{defdeltaora}. Letting $\dist(y,\partial \Omega) \ra +\infty$ along $V_j$ and using \eqref{CHR}, we deduce that $\delta(y) \ra 0$ uniformly as $y$ diverges in $V_j$. By the very definition of $\delta_j(y)$ and our assumption on $f$, this implies $u(y) \ra \lambda$ uniformly as $\dist(y, \partial \Omega) \ra +\infty$ along $V_j$. This also implies that $u^* =\lambda$, and concludes the proof.
\end{proof}
\begin{remark}\label{maxprinc}
\emph{If $u$ solves $-\Delta u = f(u)$, in the sole assumptions 
$$
f>0 \ \text{ on } \, (0,\lambda), \qquad f(\lambda)=0, \qquad 0 \le u \le \lambda \ \text{ on } \, \Omega, 
$$
and $u \not \equiv 0$, $u \not \equiv \lambda$, then $0<u<\lambda$ on $\Omega$, by the strong maximum principle.
}
\end{remark}

To deal with monotonicity properties of solutions, we shall investigate good Killing fields more closely. We begin with the next simple observation:
\begin{lemma}\label{lem_stupidkilling}
Let $(M, \metric)$ be a complete Riemannian manifold, and let $\Omega \subseteq M$ be an open subset with non-empty boundary. Suppose that $X$ is a Killing vector field on $M$, with associated flow $\Phi_t$. Then, the next two conditions are equivalent:
\begin{itemize}
\item[$(i)$] there exists $x \in \partial \Omega$  such that $\dist\big(\Phi_t(x), \partial \Omega\big) \ra +\infty$ as $t\ra +\infty$;
\item[$(ii)$] $\dist\big(\Phi_t(y), \partial \Omega\big) \ra +\infty$ locally uniformly for $y \in \partial \Omega$.
\end{itemize}
\end{lemma}
\begin{proof}
Indeed, suppose that $(i)$ holds and let $y \in \partial \Omega$. Since $\Phi_t$ is a flow of isometries, $\dist(x,y) = \dist(\Phi_t(x), \Phi_t(y))$. For each $p_t \in \partial \Omega$ realizing $\dist(\Phi_t(y), \partial \Omega)$, by the triangle inequality we thus get
$$
\begin{array}{lcl}
\disp \dist\big(\Phi_t(x),\partial \Omega\big) \, - \, \dist(x,y) & \le &  \dist\big(\Phi_t(x),p_t\big)\, - \, \dist(x,y) \\[0.2cm]
& \le & \disp \dist\big(\Phi_t(y),p_t\big) = \dist\big(\Phi_t(y),\partial \Omega\big),    
\end{array}
$$
from which $(ii)$ immediately follows.
\end{proof}
\begin{remark}
\emph{Note that condition $(ii)$ in Definition \ref{def_goodkilling}, together with Lemma \ref{lem_stupidkilling}, implies that a good Killing vector field $X$ is nowhere vanishing on $\overline \Omega$.
}
\end{remark}
As anticipated, in the presence of a good Killing field on $\Omega$, and for suitable nonlinearities $f$, we can construct a strictly monotone, non-constant solution of
\begin{equation}\label{ubordo2}
\left\{\begin{array}{l}
-\Delta u = f(u) \qquad \text{on } \Omega, \\[0.2cm]
u>0 \quad \text{on } \Omega, \qquad u=0 \quad \text{on } \partial \Omega.
\end{array}\right.
\end{equation}
\begin{proposition}\label{prop_constrsol}
Let $(M, \metric)$ with Ricci tensor satisfying $\Ricc \ge -(m-1)H^2 \metric$, let $\Omega \subseteq M$ be an open, connected set with $C^3$-boundary, and let $f \in C^1(\R)$ with the properties
\begin{equation}\label{fbecani3}
\left\{\begin{array}{rl}
{ \rm (I)} & \quad \disp f >0 \quad \text{on } (0, \lambda), \qquad f(0)= f(\lambda)=0, \\[0.3cm]
{\rm (II)} & \quad \disp f(s) \ge \left(\delta_0 + \frac{(m-1)^2H^2}{4}\right)s \quad \text{for } s \in (0, s_0),
\end{array}\right.
\end{equation}
for some $\lambda>0$ and some small $\delta_0,s_0>0$. Suppose that $\Omega$ admits a good Killing field $X$ transverse to $\partial \Omega$, with flow $\Phi : \R_0^+ \times \overline\Omega \rightarrow \overline\Omega$, and suppose further that 
\begin{equation}\label{proprieXflow}
\Phi(\R^+_0 \times \partial \Omega) \equiv \overline \Omega.
\end{equation}
Then, there exists a non-constant solution $u \in C^2(\Omega) \cap C^0(\overline \Omega)$ of \eqref{ubordo2} satisfying $0 < u < \lambda$ and the monotonicity $\langle \nabla u, X \rangle >0$ on $\Omega$.
\end{proposition}
\begin{remark}
\emph{Condition (II) in \eqref{fbecani3} in only required in order to show that the constructed solution is not identically zero.}
\end{remark}

\begin{remark}\label{remtopo}
\emph{The validity of \eqref{proprieXflow} and the connectedness of $\Omega$ imply that also $\partial \Omega$ is connected.
}
\end{remark}
\begin{remark}
\emph{
Property \eqref{proprieXflow} is not automatic for good Killing fields. As a counterexample, consider $M= \R^m$ with coordinates $(x', x_m) \in \R^{m-1} \times \R$, and set
$$
\Omega = \R^m \backslash \left\{ x = (x',x_m) \in \R^m \ : \ |x'|<1, \ x_m \le - (1-|x'|^2)^{-1} \right\}.
$$
Clearly, $X = \partial/\partial x_m$ is a good vector field on $\Omega$, transverse to $\partial \Omega$, but $\Phi(\R^+_0\times \partial \Omega)$ only covers the portion of $\Omega$ inside the cylinder $\{(x', x_m) : |x'|<1\}$.
}
\end{remark}

The proof of the above proposition relies on the classical sliding method in \cite{berenire}, \cite{bercaffnire}. For the convenience of the reader, we postpone it to Appendix 2.\par 
The control on the asymptotic behaviour of $u$ as $\dist(x, \partial \Omega) \ra +\infty$ ensured by Proposition \ref{prop_elegante}, coupled with the existence of a good Killing field enables us to proceed along the lines in \cite{ambrosiocabre, Arma} to obtain a sharp energy estimate. 
\begin{theorem}\label{teo_energy}
Let $(M^m, \metric)$ be a complete Riemannian manifold with $\Ricc \ge -(m-1)H^2 \metric$, for some $H>0$, and let $f \in C^1(\R)$ with the properties
\begin{equation*}
\left\{\begin{array}{l}
\disp f >0 \quad \text{on } (0, \lambda), \qquad f(\lambda)=0, \qquad f < 0 \quad \text{on } (\lambda, \lambda + s_0), \\[0.3cm]
\disp f(s) \ge \left(\delta_0 + \frac{(m-1)^2H^2}{4}\right)s \quad \text{for } s \in (0, s_0),
\end{array}\right.
\end{equation*}
for some $\lambda>0$ and some small $\delta_0,s_0>0$. Let $\Omega \subseteq M$ be a connected open set with smooth boundary, and suppose that $\Omega$ supports a good Killing field $X$. 
Let $u$ be a bounded solution of 
\begin{equation}\label{ubordo}
\left\{\begin{array}{l}
-\Delta u = f(u) \qquad \text{on } \Omega, \\[0.2cm]
u>0 \quad \text{on } \Omega, \qquad u=0 \quad \text{on } \partial \Omega
\end{array}\right.
\end{equation}
with the properties that 
\begin{equation}\label{furtherpro}
\left\{\begin{array}{l}
\|u\|_{C^1(\Omega)} < +\infty; \\[0.1cm]
\langle X, \nabla u \rangle \ge 0 \qquad \text{on } \, \Omega. 
\end{array}\right.
\end{equation}
Then, there exists a positive $C= C(\|u\|_{C^1(\Omega)})$ such that
\begin{equation}\label{energy}
\int_{\Omega \cap B_R} |\nabla u|^2\di x \le C \Big[\haus^{m-1}(\partial B_R) + \haus^{m-1}(\partial \Omega \cap B_R)\Big]
\end{equation}
\end{theorem}
\begin{proof}
Set $B_R = B_R(o)$. By Corollary \ref{stablesimple}, $\langle \nabla u, X \rangle >0$ on $\Omega$. Indeed, the possibility $\langle \nabla u, X \rangle =0$ is ruled out by \eqref{ipokilling} and since $u>0$ on $\Omega$, $u=0$ on $\partial \Omega$. Define $u_t(x) = u( \Phi_t(x))$, and note that, by the first assumption in \eqref{ipokilling}, $u_t$ is defined on $\Omega$. Since $\Phi_t$ is an isometry and $X$ is Killing,
\begin{equation}\label{propbase}
\left\{\begin{array}{l}
\partial_t u_t = \langle \nabla u, X\rangle \circ \Phi_t > 0, \qquad -\Delta u_t = f(u_t), \\[0.2cm]
|\nabla u_t|^2 = |(\di \Phi_t \circ \di u)^\sharp|^2 = |\nabla u|^2 \circ \Phi_t, \\[0.2cm]
\di (\partial_t u_t) = \di\big(\langle \nabla u, X\rangle \big) \circ \di \Phi_t = \nabla \di u(X, \di \Phi_t) + \langle \nabla u, \nabla_{\di \Phi_t} X \rangle.
\end{array}\right.
\end{equation}
We claim that $\nabla u_t = \di \Phi_{-t}(\nabla u)$. Indeed, for every vector field $W$ and using that $\Phi_t$ is an isometry,
$$
\langle \nabla u_t, W \rangle = \di u_t(W) = \di u\big( \di \Phi_t(W)\big) = \langle \nabla u, \di \Phi_t(W)\rangle = \langle \di \Phi_{-t}(\nabla u), W\rangle.
$$
We thus deduce that, from \eqref{propbase} and again the Killing property, 
$$
\begin{array}{rcl}
\langle \nabla u_t, \nabla(\partial_t u_t)\rangle & = & \di (\partial_t u_t)\big(\nabla u_t) =  \nabla \di u\big(X, \di \Phi_t(\nabla u_t)\big) + \langle \nabla u, \nabla_{\di \Phi_t(\nabla u_t)}X \rangle \\[0.2cm]
& = & \disp \big[\nabla \di u(X, \nabla u)\big] \circ \Phi_t + \langle \nabla u, \nabla_{\nabla u} X\rangle = \big[\nabla \di u(X, \nabla u)\big] \circ \Phi_t; \\[0.2cm]
\disp \partial_t|\nabla u_t|^2 & = & \disp \di \big( |\nabla u|^2 \circ \Phi_t\big)(\partial_t) = \langle \nabla |\nabla u|^2, X \rangle \circ \Phi_t = \big[2 \nabla \di u(\nabla u, X)\big] \circ \Phi_t, 
\end{array}
$$ 
whence 
\begin{equation}\label{baseidentity}
\frac 12 \partial_t |\nabla u_t|^2 = \langle \nabla u_t, \nabla(\partial_t u_t)\rangle.
\end{equation}
Set for convenience
$$
E_R(t) = E_{\Omega \cap B_R}(u_t) = \frac 12 \int_{\Omega \cap B_R} |\nabla u_t|^2\di x - \int_{\Omega \cap B_R} F(u_t),
$$
where $E_{\Omega \cap B_R}$ and $F$ are as in \eqref{energia}.
In our assumptions, since $u$ is bounded, by Proposition \ref{prop_elegante} we have that $\|u\|_{L^\infty} = \lambda$ and $u(x) \ra \lambda$ uniformly as $\dist(x,\partial \Omega) \ra +\infty$ along each fixed connected component of $\Omega_{R_0}$. Using $(ii)$ of Definition \ref{def_goodkilling}, and Lemma \ref{lem_stupidkilling}, we deduce that
$$
\|u_t\|_{L^\infty} \le \lambda, \qquad u_t(x) \ra \lambda  \quad \text{as } t \ra +\infty, \text{ pointwise on } \Omega \cap B_R. 
$$
By the first assumption in \eqref{furtherpro}, there exists a uniform constant $C$ such that
\begin{equation}\label{gradut}
\|\nabla u_t\|_{L^\infty(\Omega \cap B_R)} \le C \qquad \text{for every } t \in \R^+,  
\end{equation}
whence, by elliptic estimates, up to a subsequence 
\begin{equation}\label{buonaconvergenza}
u_t \looparrowright \lambda \qquad \text{in } C^{2,\alpha}(\Omega \cap B_R). 
\end{equation}
Differentiating under the integral sign with the aid of \eqref{baseidentity}, integrating by parts and using \eqref{propbase}, \eqref{gradut} we get 
$$
\begin{array}{lcl}
\disp \frac{\di E_R(t)}{\di t} & = & \disp \int_{\Omega \cap B_R} \langle \nabla u_t, \nabla (\partial_t u_t) \rangle \di x - \int_{\Omega \cap B_R} f(u_t)(\partial_tu_t)\\[0.4cm]
& = & \disp \int_{\partial(\Omega \cap B_R)} (\partial_\nu u_t) \partial_t u_t\di \sigma - \int_{\Omega \cap B_R} (\partial_tu_t)\big[\Delta u_t + f(u_t)\big] \di x  \\[0.4cm]
& = & \disp \int_{\partial(\Omega \cap B_R)} (\partial_\nu u_t) \partial_t u_t\di \sigma \ge - C \int_{\partial (\Omega \cap B_R)} \partial_t u_t \di \sigma.
\end{array}
$$
Now, integrating on $(0,T)$ and using Tonelli's theorem we obtain
$$
\begin{array}{lcl}
\disp E_R(T) - E_R(0) & \ge & \disp - C \int_0^T \int_{\partial (\Omega \cap B_R)} \partial_t u_t \di \sigma \, \di t = \disp C\int_{\partial (\Omega \cap B_R)}\big[u_T - u_0\big] \di \sigma. \\[0.5cm]
& \ge & - 2C\lambda \haus^{m-1}\big( \partial (\Omega \cap B_R)\big) \\[0.3cm]
& \ge & \disp -2C\lambda \big[ \haus^{m-1}(\partial B_R) + \haus^{m-1}(\partial \Omega \cap B_R) \big].
\end{array}
$$
Since $F(u_t) \ge 0$, we deduce 
$$
\int_{\Omega \cap B_R} |\nabla u|^2 \di x \le E_R(0) \le E_R(T) +2C\lambda \big[ \haus^{m-1}(\partial B_R) + \haus^{m-1}(\partial \Omega \cap B_R) \big].
$$
By \eqref{buonaconvergenza}, $E_R(T) \ra 0$ ar $T \looparrowright +\infty$, fron which the desired estimate \eqref{energy} follows.
\end{proof}
Putting together with Theorem \ref{teo_mainbordo}, we easily prove Theorem \ref{teo_speciale}.
\begin{proof}[\textbf{Proof of Theorem \ref{teo_speciale}}]
By Corollary \ref{stablesimple}, either $\langle \nabla u, X \rangle > 0$ or $\langle \nabla u, X \rangle = 0$ on $\Omega$. However, from the existence of $o \in \partial \Omega$ with the property $(ii)$ of Definition \ref{def_goodkilling}, and since $u>0$ on $\Omega$ and $u=0$ on $\partial \Omega$, we infer that the second possibility cannot occur. Moreover, by Hopf Lemma $\partial_\nu u >0$ on $\partial \Omega$. Via Bishop-Gromov volume estimate, assumption $\Ricc \ge 0$ imples $\haus^{2}(\partial B_R) \le 4\pi R^2$ for $R >0$, thus applying Theorem \ref{teo_energy} with $H=0$ we deduce that, by \eqref{ipovolume},
$$
\int_{\Omega \cap B_R} |\nabla u|^2 \le C\big[ 4\pi R^2 + \haus^2(\partial \Omega \cap B_R)\Big] = o(R^2 \log R)
$$
as $R \ra +\infty$. Now, the conclusion follows by applying Theorem \ref{teo_mainbordo}.
\end{proof}
To conclude, particularizing to the flat case we recover Theorem 1.8 in \cite{Arma}. 
\begin{corollary}\label{teo_specialer3}
Let $\Omega \subseteq \R^3$ be an open set with a $C^3$ boundary. Suppose that $\Omega$ can be described as the epigraph of a function $\varphi \in C^3(\R^2)$ over some plane $\R^2$, and that $\varphi$ is globally Lipschitz on $\R^2$. Let $f \in C^1(\R)$ satisfy
\begin{equation}\label{fbecanir3}
\left\{\begin{array}{l}
\disp f >0 \quad \text{on } (0, \lambda), \qquad f(\lambda)=0, \qquad f<0 \quad \text{on } (\lambda, \lambda + s_0), \\[0.2cm]
\disp f(s) \ge \delta_0s \quad \text{for } s \in (0, s_0), \\[0.2cm]
f \text{ is non-increasing on } [\lambda -s_0, \lambda],
\end{array}\right.
\end{equation}
for some $\lambda>0$ and some small $\delta_0,s_0>0$. Then, if there exists a non-constant, positive, bounded solution $u \in C^{3}(\overline\Omega)$ of the overdetermined problem 
\begin{equation}\label{pser2}
\left\{ \begin{array}{ll} 
-\Delta u = f(u) & \quad \text{on } \Omega \\[0.1cm]
u= 0 & \quad \text{on } \partial \Omega \\[0.1cm]
\partial_\nu u = \, \mathrm{constant } & \quad \text{on } \partial \Omega,
\end{array}\right.
\end{equation}
$\Omega$ is an half-space and, up to translations, $\partial \Omega = v^\perp$ for some $v \in \esse^2$, $u(x) = y( \langle x, v\rangle)$ and $y'' = -f(y)$ on $\R^+$.
\end{corollary}

\begin{proof}
Up to an isometry, we can assume that $\R^3 = \R^2 \times \R$ with coordinates $(x',x_3)$, and that 
$$
\Omega = \big\{(x', x_3) \, : \, x' \in \R^2, \, x_3 > \varphi(x') \big\}. 
$$
Let $X = \partial/\partial x_3$ be the translational vector field along the third coordinate direction, and let $\Phi_t$ be the associated flow. Clearly, $\dist(\Phi_t(x), \partial \Omega) \ra +\infty$ as $t \ra +\infty$, thus $X$ is a good Killing field for $\Omega$. By Theorem 1.1 in \cite{bercaffnire}, the monotonicity $\langle \nabla u, X\rangle >0$ is satisfied on $\Omega$. Having fixed an origin $o \in \partial \Omega$, since $\varphi$ is globally Lipschitz we deduce that
$$
\haus^{2}(\Omega \cap B_R) \le \int_{B_R \subseteq \R^2}\sqrt{1+|\nabla \varphi(x')|^2}\di x' \le CR^2 \qquad \text{as } \, R \ra +\infty.
$$
Therefore, the desired conclusion follows from Theorem \ref{teo_speciale}.
\end{proof}

\section*{Appendix 1: some remarks on $L^\infty$ bounds for $u$ and $\nabla u$.}
Under some mild conditions on the nonlinearity $f(t)$, it can be proved that both $u$ and $|\nabla u|$ are globally bounded on $M$. In this appendix, we collect and comment on two general estimates. We first examine $L^\infty$ bounds for $u$. Suppose that $\Delta $ satisfies the strong maximum principle at infinity (also called the Omori-Yau maximum principle), briefly (SMP). We recall that, by definition, $\Delta $ satisfies (SMP) if, for every $w \in C^2(M)$ with $w^* = \sup w < +\infty$, there always exist a sequence $\{x_k\} \subseteq M$ such that
$$
w(x_k) > u^* -\frac 1k, \qquad |\nabla w|(x_k) < \frac 1 k, \qquad \Delta  w(x_k) < \frac 1 k.
$$
At it is shown in \cite{prsmemoirs}, (SMP) turns out to be an extremely powerful tool in modern Geometric Analysis, and its validity is granted via mild function-theoretic properties of $M$. In particular, if $r(x)$ denotes the distance from a fixed point, the conditions
\begin{equation}\label{condricc}
\begin{array}{lcl}
\Ricc (\nabla r,\nabla r)(x) \ge -(m-1)G(r(x)),  \\[0.2cm]
G(t) = Ct^2 \log^2 t  \quad \text{for } t>>1, 
\end{array}
\end{equation}
where $C>0$ and $G$ is a smooth, positive and non-decreasing function defined on $[0,+\infty)$, ensure that (SMP) holds for $\Delta $.  A proof of this fact can be found, for instance, in \cite{prsmemoirs}, Example 1.13. Observe that \eqref{condricc} includes the cases 
$$
(i) \ \ \Ricc \ge -(m-1)H^2 \metric \qquad \text{and} \qquad (ii) \ \ K \ge -H^2, \qquad \text{for some constant } H\ge 0,
$$
which have originally been investigated by S.T. Yau 
(case $(i)$, in \cite{chengyau75, chengyau77}) and H. Omori (case $(ii)$, in \cite{omori}). Under the validity of (SMP), the next general result enables us to obtain $L^\infty$ bounds for wide classes of differential inequalities.
\begin{theorem}[\cite{prsmemoirs}, Theorem 1.31]\label{teo_allamotomya}
Suppose that $\Delta $ satisfies (SMP), and let $u \in C^2(M)$ be a solution of $\Delta  u \ge - f(u)$, for some $f \in C^0(\R)$. Then, 
$$
u^* = \sup_M u < +\infty \qquad \text{and} \qquad  f(u^*) \ge 0 
$$
provided that there exists a function $F$, positive on $[a, +\infty)$ for some $a \in \R$, with the following properties:
\begin{equation}\label{condimotomya}
\left\{ \int_a^t F(s) \di s\right\}^{-1/2} \in L^1(+\infty), \qquad \limsup_{t\ra +\infty} \frac{ \int_a^t F(s) \di s}{tF(t)} < +\infty, \qquad \liminf_{t\ra +\infty} \frac{-f(t)}{F(t)} >0.
\end{equation}
\end{theorem}

Next, we consider $L^\infty$ bounds for $\nabla u$, where $u\in C^3(M)$ is a bounded solution of $-\Delta u = f(u)$ and $f \in C^1(\R)$. 

\begin{remark}\label{gradientbound}

\rm  Since $ \Ricc \ge 0$ and $M$ is complete, the property : $u \in L^\infty(M) \Rightarrow \nabla u \in L^\infty(M)$ holds true as a (quite standard) consequence of  the Bochner identity and the De Giorgi-Nash-Moser's regularity theory for PDE's (cfr. for instance \cite{Pli, Saloff-Coste}). Indeed, if $u$ solves $-\Delta u = f(u)$ then, by Bochner formula, 
$$
\Delta |\nabla u|^2 \ge -2f'(u)|\nabla u|^2 + 2 \Ricc(\nabla u, \nabla u) \ge -C |\nabla u|^2 \qquad \text{on } M.
$$
Now, we can run the Moser iteration to get the desired bound (since $f(u)$ is bounded on $M$). In fact, we have used the well-known facts that any Riemannian manifold with $\Ricc \ge 0$ has a scale invariant $L^2$ Neumann Poincar\'e inequality and a relative volume comparison property.  

\noindent We conclude this remark by pointing out that the property: $u \in L^\infty(M) \Rightarrow \nabla u \in L^\infty(M)$ holds true also under the less restrictive assumption  $\Ricc \ge -(m-1)H^2 \metric$, for some $H\ge 0$ (cfr. for instance \cite{mastroliarigoli}, Theorem 2.6 and Corollary 2.7). We stress that the techniques to prove this generalized result are different from the ones outlined above, and rely on Ahlfors-Yau type gradient estimates. 

\end{remark}
%

%

%

%

As a prototype case, we now prove uniform $L^\infty$ bounds for $u$ and $\nabla u$ for the Allen-Cahn equation appearing in De Giorgi's conjecture. 
\begin{corollary}
Let $M$ be a complete manifold satisfying $\Ricc \ge -(m-1)H^2 \metric$, for some $H\ge 0$, and let $u \in C^2(M)$ be a solution of the Allen-Cahn equation
$$
-\Delta u = u-u^3 \qquad \text{on } M.
$$
Then, $u$ is smooth, $-1 \le u(x) \le 1$ for every $x \in M$ and $|\nabla u| \in L^\infty(M)$.
\end{corollary}
\begin{proof} By standard elliptic estimates $u$ is smooth on $M$. 
In our assumptions on the Ricci curvature, by the remarks above $M$ satisfies (SMP). Set $f(t)=t-t^3$. It is easy to check that $F(t)=t^3$ satisfies the assumptions in \eqref{condimotomya}. Then, by Theorem \ref{teo_allamotomya}, $u$ is bounded above and $(u^*)^3-u^* \le 0$, which gives $u^* \in [0,1]$ or $u^* \le -1$. Analogously, the function $w=-u$ satisfies 
$$
\Delta w = -\Delta u = u-u^3 = w^3-w,
$$
and applying the same result we deduce that either $w^* \in [0,1]$ or $w^* \le -1$. Since $w^*=-u_* : = - \inf_M u$ we deduce that either $u_* \in[-1,0]$ or $u_* \ge 1$. Combining with the above estimates for $u^*$ the $L^\infty$ bound for $u$ follows immediately. The $L^\infty$ bound for $\nabla u$ is a direct consequence of Remark \ref{gradientbound}.  
\end{proof}

\begin{remark}
\emph{In the Euclidean case, all the distributional solutions~$u\in L^1_\loc(\R^m)$
of the Allen-Cahn equation $-\Delta u=u - u^3$
(and more generally of the vector valued Ginzburg-Landau equation~$-\Delta u=u(1-|u|^2)$)
always satisfy the bound $|u|\le 1$, see Proposition~1.9 in~\cite{FTh}. Hence, by standard elliptic estimates, they are smooth and all their derivatives are bounded too. }\end{remark}

\section*{Appendix 2: construction of a monotone solution}
In this appendix, under the presence of a good Killing field on $\Omega$, we construct a non-constant solution of
\begin{equation}\label{ubordo=app}
\left\{\begin{array}{l}
-\Delta u = f(u) \qquad \text{on } \Omega, \\[0.2cm]
u>0 \quad \text{on } \Omega, \qquad u=0 \quad \text{on } \partial \Omega.
\end{array}\right.
\end{equation}
We recall the geometric assumptions: let $(M, \metric)$ with Ricci tensor satisfying $\Ricc \ge -(m-1)H^2 \metric$, let $\Omega \subseteq M$ be an open set with $C^3$-boundary, and let $f \in C^1(\R)$ with the properties
\begin{equation}\label{fbecani4}
\left\{\begin{array}{rl}
{\rm (I)} & \quad \disp f >0 \quad \text{on } (0, \lambda), \qquad f(0)=f(\lambda)=0, \\[0.3cm]
{\rm (II)} & \quad \disp f(s) \ge \left(\delta_0 + \frac{(m-1)^2H^2}{4}\right)s \quad \text{for } s \in (0, s_0),
\end{array}\right.
\end{equation}
for some $\lambda>0$ and some small $\delta_0,s_0>0$. Suppose that $X$ is a good Killing field on $\Omega$, with flow $\Phi : \R_0^+ \times \overline\Omega \rightarrow \overline\Omega$. 
\begin{proposition}\label{prop_constrsol2}
In the above assumptions, suppose further that $X$ is transverse to $\partial \Omega$ and that
\begin{equation}\label{assunzflowX}
\Phi(\R^+_0 \times \partial \Omega) \equiv \overline \Omega.
\end{equation}
Then, there exists a non-constant solution $u \in C^2(\Omega) \cap C^0(\overline \Omega)$ of \eqref{ubordo=app} such that $0<u<\lambda$ and $\langle \nabla u, X \rangle >0$ on $\Omega$.
\end{proposition}
\begin{proof}
Let $\{U_j\} \uparrow \partial \Omega$ be a smooth exhaustion of $\partial \Omega$. By the properties of flows and the transversality of $X$ and $\partial \Omega$, the map $\Phi$ restricted to $\R^+ \times U_j$ realizes a diffeomorphism onto its image. We briefly prove it. To show that $\Phi$ is injective, suppose that there exist $(t_1,x_1) \neq (t_2,x_2)$ for which $\Phi(t_1,x_1) = \Phi(t_2,x_2)$. Then, by the properties of the flow, necessarily $t_1 <t_2$ (up to renaming). Since $\Phi(t_2,x_2) = \Phi(t_1, \Phi_{t_2-t_1}(x_2))$, the equality and the fact that $\Phi_{t_1}$ is a diffeomorphism imply that $x_1= \Phi_{t_2-t_1}(x_2)$. Hence, the flow line $\Phi_t(x_1)$ intersects twice the boundary $\partial \Omega$, and by property $(i)$ of good Killing fields it holds $\Phi_{|[0,t_2-t_1]}(x_1) \subseteq \partial \Omega$, which is impossible since $X$ is transverse to $\partial \Omega$. Next, we show that $\di \Phi$ is nonsingular. Indeed, if at a point $\Phi(t,x)$ we have $X_{
 \Phi(t,x)}= \di \Phi(\partial_t)= \di \Phi(Z_x)$ for some nonzero $Z_x \in T_x\partial\Omega$, then applying $\di \Phi_{-t}$ we would have $X_x = Z_x$, which is impossible again by the transversality of $X$ and $\partial \Omega$. Next, choose a sequence $\{T_k\} \uparrow +\infty$ and define the cylinders $C_{jk} = \Phi([0,T_k] \times U_j)$. By \eqref{assunzflowX}, $\Omega = \bigcup_{j,k}C_{jk}$ (this is the only point where \eqref{assunzflowX} is used). Denote with $\pi_1 : \R^+_0 \times \overline U_j \ra \R^+_0$ the projection onto the first factor, and with $\pi = \pi_1 \circ \Phi^{-1} : \overline C_{jk} \ra \R^+_0$ its image through $\Phi$. Take a sequence $\{\psi_k\} \subseteq C^\infty(\R^+_0)$ with the following properties:
$$
\begin{array}{l}
0 \le \psi_k \le \lambda \ \text{on } \R^+, \quad \psi_k \equiv \lambda \ \text{ on } [T_k, +\infty), \quad \psi_k(0)=0 \ \text{ for each } k, \\[0.1cm]
\psi_k \ \text{ is strictly monotone on } [0, T_k], \quad \psi_k \ge \psi_{k+1} \ \text{ on } \R^+_0.
\end{array}
$$
For every pair $(j,k)$, let $u_{jk} \in C^2(C_{jk}) \cap C^0(\overline{C}_{jk})$ be a solution of
\begin{equation}\label{ujk}
\left\{ 
\begin{array}{l}
-\Delta u_{jk} = f(u_{jk}) \qquad \text{on } C_{jk} \\[0.2cm]
u_{jk} = \psi_k \circ \pi  \qquad \text{on } \Phi\big([0,T_k] \times \partial U_j\big) \\[0.2cm]
u_{jk} = 0 \quad \text{on } \Phi\big(\{0\} \times U_j\big) \qquad u_{jk} = \lambda \quad \text{on } \Phi\big(\{T_k\} \times U_j\big), 
\end{array}\right.
\end{equation}
constructed via the monotone iteration scheme (see \cite{sattinger}) by using $u \equiv 0$ as a subsolution and $u\equiv \lambda$ as a supersolution. Then, $0 \le u_{jk} \le \lambda$ on $\overline C_{jk}$, and the inequality is strict on $C_{jk}$ by the strong maximum principle (see Remark \ref{maxprinc}).\\
\textbf{Step 1: $u_{jk}$ is monotone in $t$ on $C_{jk}$}.\\
\noindent To prove this claim, for $t \in \R^+$ set 
$$
w_t = u_{jk}\circ \Phi_{-t} - u_{jk} \qquad \text{on } \, V_{t} = \Phi_{t}(C_{jk}) \cap C_{jk} = \left\{ \begin{array}{ll}
\emptyset & \quad \text{if } t \ge T_k \\[0.2cm]
\Phi\big(U_j \times (t, T_k)\big) & \quad \text{if } t \in [0, T_k). \end{array}\right. 
$$
Hereafter we omit writing the pair $(j,k)$. In our assumptions, for every $t >0$ 
\begin{equation}\label{eqwt}
\left\{ \begin{array}{l} -\Delta w_t = c_t(x)w_t \qquad \text{on } V_t \\[0.2cm]
w_t <0 \qquad \text{on } \partial V_t, 
\end{array}\right. \qquad \text{where} \quad c_t(x) = \frac{f(u_{jk} \circ \Phi_{-t}) - f(u_{jk})}{w_t}
\end{equation}
We now claim that, if $t$ is sufficiently close to $T_k$, then the operator $L_t = \Delta +c_t(x)$ is non-negative on $V_t$ (as observed by S.R.S. Varadhan and A. Bakelman, see Proposition 1.1 of \cite{berenire}). Indeed let $S>0$ be the $L^2$-Sobolev constant of $W = \Phi\big( [0, 2T_k] \times U_j\big)$: 
$$
S\|\phi\|_{L^{2^*}(W)} \le \|\nabla \phi\|_{L^2(W)} \qquad \text{for every } \, \phi \in C^\infty_c(W).
$$
Then, for every $V \subseteq W$ and every $\phi \in C^\infty_c(V)$, by Cauchy-Schwarz inequality
$$
\int c_t\phi^2 \le \lip_{[0,\lambda]}(f) \int \phi^2 \le \lip_{[0,\lambda]}(f)|V|^\frac{2}{m}\left( \int \phi^{\frac{2m}{m-2}}\right)^{\frac{m-2}{m}} \le \frac{\lip_{[0,\lambda]}(f)|V|^\frac{2}{m}}{S} \int |\nabla \phi|^2,
$$
where $\lip_{[0,\lambda]}(f)$ is the Lipschitz constant of $f$ on $[0, \lambda]$. If $|V|$ is sufficiently small (and the bound does not depend on $t \in (0, T_k]$), it thus follows that 
$$
\int |\nabla \phi|^2 - \int c_t \phi^2 \ge 0,
$$
which means that $L_t$ has non-negative spectrum. Particularizing to $V=V_t$ proves the claim.  By a classical result, \cite{bernirevara}, the non-negativity of $L_t$ on $V_t$ is equivalent to the validity of the maximum principle for $L_t$ on $V_t$, hence, by \eqref{eqwt}, $w_t \le 0$ on $V_t$. The strong maximum principle then imples the strict inequality $w_t<0$ on $V_t$. Now, consider 
$$
\mathcal{T} = \big\{ t \in [0, T_k] \, : \, w_s <0 \, \text{ on } V_s, \, \text{ for each } s \in [t,T_k]\big\}, 
$$
which by the previous claim is non-empty and contains a left neighbourhood of $T_k$. We are going to prove that $\bar t = \inf \mathcal{T}= 0$. If, by contradiction, $\bar t > 0$, then by continuity $w_{\bar t} \le 0$ on $V_{\bar t}$. Since,  by \eqref{eqwt},  $(w_{\bar t})_{|\partial V_{\bar t}} <0$, the strong maximum principle implies that $w_{\bar t} <0$ on $\overline V_{\bar t}$. By compactness, let $\eps>0$ be such that $w_{\bar t} < -\eps$ on $\overline V_{\bar t}$, and by continuity choose $\eta>0$ sufficiently small in order to satisfy the next requirements:
\begin{itemize}
\item[-] the operator $L_{\bar t-\eta}$ is non-negative on $V_{\bar t-\eta}\backslash V_{\bar t} = \Phi\big( (\bar t-\eta, \bar t] \times U_j\big)$;
\item[-] $w_{\bar t -\eta} \le -\dfrac \eps 2$ on $V_{\bar t}$.
\end{itemize}
By our construction, $w_{\bar t -\eta} <0$ on $\partial(V_{\bar t-\eta}\backslash V_{\bar t})$, thus by the maximum principle $w_{\bar t-\eta} <0$ on $V_{\bar t-\eta}\backslash V_{\bar t}$ and so on $V_{\bar t-\eta}= V_{\bar t} \cup (V_{\bar t-\eta}\backslash V_{\bar t})$, contradicting the minimality of $\bar t$. Concluding, $\bar t =0$, hence $w_t >0$ on $\R^+$ for every $t \in (0, T_k]$, which proves the monotonicity of $u$ in the $t$-direction.\\
\noindent \textbf{Step 2: the limiting procedure}.\\
First, by requirement $(ii)$ in Definition \ref{def_goodkilling} of a good Killing field we argue that $\Omega_{R_0} = \{ x  \in \Omega : \dist(x,\partial \Omega) \ge R_0\}$ is non-empty for each $R_0$. In our assumptions on $\Ricc$ and on $f$, by a comparison procedure identical to that performed in Lemma \ref{lem_infnonzero} we can find:
\begin{itemize}
\item[-] $R>0$ such that $\lambda_1(\bh_R) s < f(s)$ for every $s \in [0, s_0]$, where $\lambda_1(\bh_R)$ is the first eigenvalue of a geodesic ball $\bh_R$ in a space form $M_H$;
\item[-] $y \in \Omega_{2R_0}$ and a Lipschitz, weak solution $w \in \lip(B_R(y))$ of 
$$
\left\{\begin{array}{l}
- \Delta w \le \lambda_1(\bh_R)w < f(w) \qquad \text{on } B_R(y), \\[0.2cm]
w_{|\partial B_R(y)} = 0, \qquad w>0 \quad \text{on } B_R(y), \qquad \|w\|_{L^\infty(B_R(y))} < s_0.
\end{array}\right.
$$
\end{itemize}
We arrange the exhaustion $\{U_j\}$ in such a way that $B_R(y) \Subset \Phi(\R^+ \times U_0)$, and for each fixed $j$ we let $k=k_j$ be such that $B_R(y) \Subset C_{jk}$ for every $k \ge k_j$. This latter property is possible by $(ii)$ of Definition \ref{def_goodkilling} of a good Killing field, together with Lemma \ref{lem_stupidkilling}. Since $\partial \Omega \in C^3$, we can smooth the corners of $C_{jk}$ in such a way that $\partial C_{jk} \in C^3$. By uniform elliptic estimates, up to passing to a subsequence $\{u_{jk}\}_k$ converges in $C^{2,\alpha}_\loc$ to a solution $u_j$ of
$$
- \Delta u_j = f(u_j) \quad \text{on } C_j = \Phi(\R^+_0 \times U_j), \qquad u_j = 0 \, \text{ on } \Phi\big(\{0\} \times U_j\big), \qquad 0 \le u_j \le \lambda.
$$
Moreover, by comparison $u_{jk} \ge w$ on $B_R(y)$, hence $u_j \ge w$. Letting now $j \ra +\infty$ and using again elliptic estimates we get the existence of the desired $u$ with $0\le u \le \lambda$. From $u_j \ge w$ we deduce that $u \ge w$, thus $u$ is non-zero. By Remark \ref{maxprinc} and since $u=0$ on $\partial \Omega$, $0 <u<\lambda$ on $\Omega$. The monotonicity relation $\langle \nabla u, X\rangle \ge 0$ follows from that of $u_{jk}$ via pointwise convergence. To prove the stronger $\langle \nabla u, X \rangle >0$, we apply Corollary \ref{stablesimple} to get that either $\langle \nabla u, X\rangle \equiv 0$ or $\langle \nabla u, X\rangle >0$. The first case is ruled out, because it would mean that $u$ is constant on the flow lines of $X$: starting from a point $x \in \partial \Omega$, this and the positivity of $u$ on $\Omega$ would imply that $\Phi_t(x) \in \partial \Omega$ for every $t \in \R^+$, contradicting property $(ii)$ of Definition \ref{def_goodkilling} (or, even, contradicting the trasversality of $X$ and $\partial \Omega$).
\end{proof}
\vspace{0.5cm}
\noindent \textbf{Acknowledgements: } The second author is indebted to Jorge Herbert de Lira for a stimulating discussion about Killing fields on Riemannian manifolds.
The first and the third
authors were supported by the
ERC grant EPSILON ({\it Elliptic Pde's and Symmetry of Interfaces and Layers
for Odd Nonlinearities}).

\bibliographystyle{plain}


%
%

%
%

\end{document}